\documentclass[a4paper]{amsart} 
\usepackage{amsmath}
\usepackage{amssymb}
\usepackage{verbatim}
\usepackage{amsthm}
\usepackage{mathrsfs}
\usepackage{graphicx}
\usepackage{mathtools}
\usepackage[colorlinks,pdfdisplaydoctitle,linkcolor=blue,citecolor=red]{hyperref}
\usepackage{caption}
\usepackage{subcaption}
\usepackage{url}
\usepackage{epstopdf}
\usepackage{pgf,tikz}
\usepackage{todonotes}
\usepackage{bbm}
\usepackage{extarrows}
\usetikzlibrary{arrows}
\usepackage{algorithm}
 \usepackage{algorithmicx}
\usepackage{dsfont}
\usepackage{listings}
\usepackage{color}
\usepackage{enumerate}
 \usepackage{relsize}
\usepackage{xfrac}

\renewcommand{\Delta}{\triangle}

\definecolor{darkblue}{rgb}{0,0,0.7}

\definecolor{darkgreen}{rgb}{0.01,0.75,0.24}

\def \Ee[#1]{\mathcal{E}^{\text{{#1}}}}
\def\R{\mathbf{R}}

\def\pa[#1,#2]{\frac{\partial {#1}}{\partial {#2}} }

\def\idom[#1,#2,#3]{\int_{#1}\hspace{1pt} {#2} \hspace{1pt} \text{d}{#3}}
\def\res[#1,#2]{\left.{#1}\right|_{#2}}

\def\var[#1,#2]{\langle \delta \mathcal{E}^{\text{{#1}}}({#2}),v\rangle}
\def\vars[#1,#2,#3]{\langle \delta^2\mathcal{E}^{\text{{#1}}}({#2})v,{#3}\rangle}
\def\vard[#1,#2,#3,#4]{\langle \delta\mathcal{E}^{\text{{#1}}}({#2})-\delta\mathcal{E}^{\text{{#3}}}({#4}),v\rangle}

\def\N{\mathbb{N}}

\newcommand{\cO}{\mathcal{O}}

\newcommand{\cA}{\mathcal{A}}
\newcommand{\cB}{\mathcal{B}}

\newcommand{\cG}{\mathcal{G}}

\newcommand{\cI}{\mathcal{I}}

\newcommand{\cP}{\mathcal{P}}

\newcommand{\cX}{\mathcal{X}}

\DeclareMathOperator*{\spann}{span}

\newcommand{\be}{\begin{equation}}
\newcommand{\en}{\end{equation}}
\newcommand{\ben}{\begin{equation*}}
\newcommand{\enn}{\end{equation*}}
\newcommand{\bea}{\begin{aligned}}
\newcommand{\ena}{\end{aligned}}

\def\ba#1\ena{\begin{align}#1\end{align}}
\def\ban#1\enan{\begin{align*}#1\end{align*}}

\theoremstyle{plain}
\newtheorem{thm}{Theorem}[section]

\newtheorem{lem}[thm]{Lemma}

\newtheorem{cor}[thm]{Corollary}

\newtheorem{proposition}[thm]{Proposition}

\newtheorem{remark}[thm]{Remark}
\newtheorem{example}[thm]{Example}

\definecolor{darkred}{RGB}{139,0,0}
\definecolor{darkgreen}{RGB}{0,100,0}
\definecolor{darkmagenta}{RGB}{139,0,139}
\definecolor{darkpurple}{RGB}{110,0,180}
\definecolor{darkblue}{RGB}{40,0,200}
\definecolor{darkorange}{RGB}{255,140,0}

\numberwithin{equation}{section}

\begin{document}

\title[Box-constrained ensemble Kalman inversion]{On the incorporation of box-constraints for ensemble Kalman inversion}
\author[N. K. Chada] {Neil K. Chada}
\address{Department of Statistics and Applied Probability, National University of Singapore, 119077, Singapore}
\email{neil.chada@nus.edu.sg}

\author[C. Schillings] {Claudia Schillings}
\address{Institute of Mathematics, University of Mannheim, 68131 Mannheim, Germany}
\email{c.schillings@uni-mannheim.de}

\author[S. Weissmann] {Simon Weissmann}
\address{Institute of Mathematics, University of Mannheim, 68131 Mannheim, Germany}
\email{sweissma@mail.uni-mannheim.de}

\begin{abstract}
The Bayesian approach to inverse problems is widely used in practice to infer unknown parameters from noisy observations. In this framework, the ensemble Kalman inversion has been successfully applied for the quantification of uncertainties in various areas of applications. In recent years, a complete analysis of the method has been developed for linear inverse problems adopting an optimization viewpoint. However, many applications require the incorporation of additional constraints on the parameters, e.g. arising due to physical constraints. We propose a new variant of the ensemble Kalman inversion to include box constraints on the unknown parameters motivated by the theory of projected preconditioned gradient flows. Based on the continuous time limit of the constrained ensemble Kalman inversion, we discuss a complete convergence analysis for linear forward problems. {We adopt techniques from filtering, such as variance inflation, which are crucial in order to improve the performance and establish a correct descent.} These benefits are highlighted through a number of numerical examples on various inverse problems based on partial differential equations.
\end{abstract}

\maketitle

\section{Introduction}
\label{sec:intro}
Inverse problems aim to recover a quantity of interest from perturbed noisy measurements. {One very common approach to infer the unknown parameters tries to minimize the difference of the measurements and the outcome of the model in a suitable norm \cite{EHNR96}.} Another more recent approach for solving inverse problems is to take a statistical approach where the quantity of interest is a probabilistic distribution constructed via Bayes' Theorem, known as the Bayesian approach, see e.g. \cite{KS04,AMS10}. The advantage of quantifying inverse problems through the latter is that it aids in quantifying uncertainty through statistical properties. Common algorithms designed for Bayesian inversion include sampling based methods such as Markov chain Monte Carlo (MCMC) methods but also variational Bayes' methods. One recent method aimed at solving Bayesian inverse problems through data assimilation methodologies \cite{LSZ15,SR19} is ensemble Kalman inversion (EKI).

{EKI can be viewed as the application of  the ensemble Kalman filter (EnKF) \cite{GE09,GE03,KLS14,GMT11} towards inverse problems. It was first proposed by Iglesias et al. \cite{ILS13,LR09}, offering a cheaper approximation of the solution compared to traditional methods.} Since its formulation a number of research directions have been considered such as applications, building theory and applying uncertainty quantification techniques \cite{BSW18,BSWW19,CIRS17,MAI16, SS17}. However, the basic version of the EKI does not allow to incorporate additional constraints on the parameters, which often arise in many applications due to additional knowledge on the system. We will focus in the following on the efficient incorporation of box constraints. An example of this type of constrains includes the case of hierarchical EKI \cite{CIRS17} where the hyperparameters are often defined through a uniform distribution. It is well known that the EKI may lead to estimates of the unknown parameters, which are unfeasible, i.e. the estimates do not satisfy the box constraints imposed by the uniform distribution on the hierarchical parameters. As a result there is a strong motivation to study the incorporation of constraints for the EKI. 

{Introducing constrained optimization for the EnKF has been a challenge, which has drawn increasingly more attention in recent years. A literature overview of existing methods for the treatment of linear and nonlinear constraints for Kalman-based methods can be found in  \cite{ARB18,DS09}. The projection of the estimates to the feasible set is a very common approach, see e.g. \cite{KIF08,WCC09}, which can be generalized to nonlinear constraints by linearization ideas. Most of the variants are motivated by interpreting the Kalman-based updates as a solution of a corresponding optimization problem, see \cite{ ARB18} for more details. This viewpoint allows to straightforwardly include constraints on the parameters and states. In \cite{ABSS18}, the authors suggest a new approach to handle linear equality and inequality constraints for the EnKF and EKI by reparameterizing the solution of the optimization problem in the range of the covariance, i.e. by seeking the solution of the optimization problem in a subspace defined by the initial ensemble. 
In this work, we will adopt the optimization viewpoint, i.e. we will view 
the EKI as a derivative-free optimizer for the least-squares misfit functional (cp. \cite{SS17}), and motivate the incorporation of additional constraints on the unknown parameters via projection-based methods to the feasible set.} To illustrate the idea, we focus in the following on the incorporation of box-constraints, i.e. we will introduce a variant of the EKI ensuring that the estimate of the unknown parameter remains within a box of the parameter domain. Our work will build on the theory by Bertsekas \cite{DPB82,DPB15} and others \cite{SKS11,SS09} for projected preconditioned gradient methods. For inverse problems, box-constrained optimization has been applied but in a different context and for different applications, we refer the reader to \cite{HKR18}. 

 We will show that a simple projection of the EKI estimate to the feasible set will not necessarily lead to a convergent scheme. To accommodate this we propose using techniques from data assimilation such as variance inflation \cite{JLA07,JLA09,TMK16} to ensure that one can attain the correct descent direction. From this our analysis will consists of an existence and uniqueness result, a quantification of the ensemble collapse and the convergence of the residuals. We aim to compare this modified projected method to the simple projection of its continuos-time limit. {The limit leads to an ordinary differential equation with discontinuous right hand side. The presented theory is based on a smoothened version of the limit, borrowing ideas from barrier methods.}
 
In addition to the EKI based on perturbed observations, we will also consider the ensemble square root filter (ESRF) applied to inverse problems. The ESRF \cite{KM15,LDN08,TAB03} is a modification of the EnKF, but with the key difference of being deterministic as there is no inclusion of the perturbed observations. We emphasize our work will be primarily focused on the EnKF for inverse problems. We will make note of this throughout the paper whether various results can be generalized to include the ESRF. In order to justify the proposed projected EKI, we provide numerical examples demonstrating the improvements, both for the original projected EKI but also the modified projected EKI with variance inflation. \\

The main contributions of our work can be summarized as follows: 
\begin{itemize}
\item We propose a new variant of the EKI and the ESRF for inverse problems which allows to incorporate box constraints on the unknown parameters. The modification of the original EKI is motivated by viewing the algorithm as a preconditioned gradient-like flow. The scheme introduces a tailored variance inflation to ensure descent directions with respect to the misfit functional in each iteration. 
\item We derive a complete convergence analysis for the EKI with {perturbed} observations and ESRF for inverse problems based on the continuous time limit of the suggested variant in the case of a linear forward problem. 
\item The validity of the results obtained are highlighted through numerical examples. All examples will be based on partial differential equations (PDEs) with linear and nonlinear forward operators.
\end{itemize}
\subsection{Outline}
The article is structured as follows: in Section \ref{sec:bco} we provide an introduction to box-constrained optimization. We present this in a general sense before discussing its application to EKI, which we also review. Section \ref{sec:conv} is dedicated to the derivation of the projected continuous time limit and gradient flow structure. Convergence and accuracy results will be presented in the linear setting, where we also introduce the notion of variance inflation. Section \ref{sec:num} is devoted to numerical experiments on PDE-constrained inverse problems. The purpose of this section is to verify the theoretical findings on both linear and non-linear problems. Finally in Section \ref{sec:conc} we conclude with a number of remarks while providing avenues of future work.

\section{Background on EKI and box-constrained optimization}
\label{sec:bco}

The goal of computation for this work is to infer the unknown parameters $u \in \mathcal{X}$ from noisy data $y \in \R^K$ of the form
\begin{equation}
\label{eq:inv}
y = \mathcal{G}(u) + \eta, \quad \eta \sim N(0,\Gamma)\,,
\end{equation}
where $\cG: \mathcal{X} \rightarrow \R^K$ is our parameter-to-observation mapping and $K\in\N$ denotes the number of observations. {For simplicity, we will assume that the parameter space is finite dimensional, i.e. $\cX=\R^n$. Most of the results presented in the following sections can be straightforwardly generalized to the infinite-dimensional setting under suitable assumptions on the forward operator $\mathcal G$. To avoid the technicalities arising from the analysis of infinite-dimensional differential equations and optimization problems, we will work under the assumption that the parameter space is finite dimensional (possibly after applying a finite-dimensional approximation of the unknown quantity $u$).}

\subsection{Continuous time limit of the EKI}
We briefly describe the EKI algorithm below and refer to \cite{GE09,ILS13,SS17} for more details on the derivation of the method. EKI works by updating {a given ensemble of $J$ particles $\{u^{(j)}_n\}_{j=1}^J \mapsto \{u^{(j)}_{n+1}\}_{j=1}^J$} through a two-step procedure. The EKI estimation of \eqref{eq:inv} is given by the usual EnKF update formula, also known as the \textit{analysis step}
\begin{equation}
\label{eq:up}
u^{(j)}_{n+1} = u^{(j)}_n + C^{up}_n(C^{pp}_n + \Gamma)^{-1}(y^{(j)} - \mathcal{G}(u^{(j)}_n)),
\end{equation}
where we have the inclusion of perturbed observations
\begin{equation}
\label{eq:pert}
y^{(j)} = y + \eta^{(j)}.
\end{equation}
The update formula depends on empirical means and covariances defined as
\begin{align*}
\bar{u}_n = \frac{1}{J} \sum_{j=1}^{J}u^{(j)}_n, \quad \bar{\mathcal{G}}_n = \frac{1}{J} \sum_{j=1}^{J}\mathcal{G}(u^{(j)}_n), 
\end{align*}
\begin{align*}
C^{up}_n &= \frac{1}{J} \sum_{j=1}^{K}(u^{(j)}_n - \bar{u}_n)\otimes (\mathcal{G}(u^{(j)}_n) - \bar{\mathcal{G}}_n),\\ 
C^{pp}_n &= \frac{1}{J} \sum_{j=1}^{J}(\mathcal{G}(u^{(j)}_n) - \bar{\cG}_n)\otimes (\mathcal{G}(u^{(j)}_n) - \bar{\mathcal{G}}_n),
\end{align*}
{where $\otimes$ denotes the tensor product for Hilbert spaces, given by
$$
z_1 \otimes z_2: \mathcal{H}_1 \rightarrow \mathcal{H}_2, \ \textrm{with} \ q \mapsto z_1 \otimes z_2(q): = \langle z_2,q\rangle_{\mathcal{H}_2} \cdot z_2
$$
where $(\mathcal{H}_1, \langle , \cdot, \rangle_{\mathcal{H}_1})$ \& $(\mathcal{H}_2, \langle , \cdot, \rangle_{\mathcal{H}_2})$ denote the Hilbert spaces and their inner products, such that $z_1 \in \mathcal{H}_1$ and $z_2 \in \mathcal{H}_2$.} The convergence analysis will be based on the analysis of the continuous time limit for projected EKI in the linear case. This was derived for the original EKI by Schillings et al. \cite{SS17}, where the limit is given by 
{
\begin{equation}
\label{eq:limit1}
\frac{du^{(j)}}{dt} = \frac{1}{J}\sum_{k=1}^J \big\langle \mathcal{G}(u^{(k)}) - \mathcal{G}(u^{(j)}), y - \mathcal{G}(u^{(j)})\big\rangle_{\Gamma} (u^{(k)} - \bar{u})
\end{equation}
with $t\in(0,T], \ j\in \{1,\ldots,J\}$ and initial condition
\begin{equation}
\label{eq:limit1ini}
{u^{(j)}}(0) =u_0^{(j)}, \qquad j\in \{1,\ldots,J\}\,.
\end{equation}
}
{Here $\langle\cdot,\cdot\rangle_\Gamma:=\langle\cdot,\Gamma^{-1}\cdot\rangle$ denotes the weighted Euclidean inner product in $\R^K$.} Note that the limit is derived by neglecting the perturbations of the observations in each iteration, see \cite{SS17} for more details. {The analysis of the stochastic differential equation corresponding to the limit of \eqref{eq:up} will be subject to future work.}

In the case of a linear forward problem, the limit can be equivalently viewed as a preconditioned gradient flow
\begin{equation*}
\frac{du^{(j)}}{dt} =  - C(u) \nabla_u \Phi(u^{(j)};y),
\end{equation*}
with {$\| \cdot\|_\Gamma:=\|\Gamma^{-1}\cdot\|$ denoting the weighted Euclidean norm in $\R^K$,} misfit functional $\Phi(u^{(j)};y) = \frac{1}{2}\|y - \cG(u^{(j)})\|^2_{\Gamma}$ and empirical covariance $C(u)$.

\subsection{Continuous time limit of the ESRF for inverse problems}
Using a deterministic transformation of the ensemble satisfying the Kalman equations leads to the ESRF for inverse problems, see \cite{RC14}.  
Based on the results from Reich et al. \cite{BR10a,BR10b}, the continuous time limit of the ESRF for inverse problems is given by
{
\begin{align}\label{eq:limit2ini}
\frac{du^{(j)}}{dt} = \frac{1}{J}\sum_{k=1}^J \big\langle \mathcal{G}(u^{(k)}) - \mathcal{G}(u^{(j)}), y - \frac{1}{2}\mathcal{G}(u^{(j)}) - \frac{1}{2}\bar{\mathcal{G}}\big\rangle_{\Gamma} (u^{(k)} - \bar{u})
\end{align}
with $t\in(0,T], \ j\in \{1,\ldots,J\}$ and initial condition
\begin{equation*}
{u^{(j)}}(0) =u_0^{(j)}, \qquad j\in \{1,\ldots,J\}\,.
\end{equation*}
and its gradient flow structure
\begin{equation}\label{eq:limit2}
\frac{du^{(j)}}{dt} =  - C(u)\big( \nabla_u \Phi(u^{(j)};y) + \nabla_u \Phi(\bar{u};y)\big),
\end{equation}}
in the linear setting.

{Both variants of the ensemble Kalman inversion satisfy the well-known subspace property, which states that the estimate will be a linear combination of the inital ensemble members. 
\begin{lem}\cite{SS17}\label{lemma:subspace_prop}
{Let $\cA$ be the linear span of the initial ensemble $\{u_0^{(j)}\}_{j=1}^J$, i.e. {$\cA=\spann\{\{u_0^{(j)}\}_{j=1}^J\}, \ J\in\mathbb N$}, then we have that $\{u^{(j)}(t)\}_{j=1}^J \in \mathcal{A}$ for all $n\in\N$ and $t\in [0,T]$} satisfying \eqref{eq:limit1} and \eqref{eq:limit2}, respectively. 
\end{lem}
}

The preconditioned gradient flow structure viewpoint opens up the perspective to include additional constraints such as box constraints or more general, nonlinear constraints. The remaining section is devoted to the introduction of the box-constrained optimization. We will focus in the following on the linear case, i.e., assuming that $\mathcal G(\cdot)=A\cdot$ with $A\in\mathcal L(\mathcal X, \R^K)$ to illustrate the basic ideas. Then, the optimization problem consists of a linear least-squares problem, i.e. $\Phi(u;y)=\frac12 \|y-Au\|_\Gamma^2$. {We will modify the differential equations \eqref{eq:limit1} and \eqref{eq:limit2}, respectively, to include the constraints using the ideas introduced below. We will see that we can interpret the suggested modifications as a variance inflation technique. In particular, this viewpoint implies that the subspace property will not hold anymore for the modified versions of the EKI.}
 
\subsection{The linear box-constrained optimization problem}
Motivated by the optimization perspective on the EKI, we consider the following optimization problem: The objective function consists of the least-squares function
  $\Phi:\cX\to\R$ with $\Phi(u) = \frac12\|Au-y\|_\Gamma^2,\ u\in\cX$, where $A:\cX\to \R^K$. We further define the set of the constraints by 
  {
\begin{equation}\label{eq:linearconstraints}
\Omega = \{u\in\cX: \langle c_j,u\rangle+ \delta_j\le 0, j=1,\dots,m\},
\end{equation}
for  $c_j\in\R^n$, $\delta_j\in\R$, $j=1,\dots,m$.  We set $h_j(u)=\langle c_j,u\rangle+ \delta_j$, i.e. the feasible set $\Omega$ is given by $\Omega = \{u\in\cX: h_j(u)\le 0, j=1,\dots,m\}$.} The constrained optimization problem is then given by
\begin{equation}\label{eq:optprob}
\min\limits_{u\in\Omega} \Phi(u).
\end{equation}

Note that the optimization problem \eqref{eq:optprob} is convex, which implies that necessary optimality conditions are also sufficient \cite{DPB15}. Let $u^*$ be a Karush Kuhn Tucker (KKT)-point for \eqref{eq:optprob}, i.e. there exists $\lambda^*\in \R^m$ such that

\begin{itemize}
\item $u^*\in\Omega,$
\item $\lambda_j^*\ge0$ for all $j\in\{1,\dots,m\}$,
\item $\lambda_j^*(\langle c_j,u^*\rangle - \delta_j) = 0$ for all $j\in\{1,\dots,m\}$,
\item $\nabla \Phi(u^*) + \sum\limits_{j=1}^m \lambda_j^*c_j = 0$,
\end{itemize}
then $u^*$ is a global minimizer of \eqref{eq:optprob}. Box constraints are a special instance of the feasible set $\Omega$, i.e. $c_j=\pm e_j$, where $e_j$ is the $j$-th unit vector, and $\delta_j$ {correspond to} the lower or upper bounds on $u_j$.

\begin{remark}
{We note that $A^\top \Gamma^{-1}A$ is always symmetric and positive semidefinite. In case that $A^\top \Gamma^{-1}A$ is strictly positive definite}, the optimization problem \eqref{eq:optprob} is strictly convex. In particular, there exists at most one stationary point $u^*$ which is the global minimum for the problem \eqref{eq:optprob} (cp. \cite{DPB15}, Proposition 2.1.1). However, please note that the assumption on the regularity of $A^\top \Gamma^{-1}A$ is in general not satisfied for inverse problems. {Typically, the number of unknown parameters is much larger than the number of observations, i.e. $n\gg K$, in the applications of interest.}
\end{remark}

{
\begin{remark}
The condition $A^\top \Gamma^{-1}A >0$ is well-known in the context of data assimilation related to as the strong observability condition. As a result this condition for the case of ensemble Kalman filter has been well-documented, see e,g, \cite{BM18,LSZ15}.
\end{remark}}

\subsubsection{{The projected gradient method for box-constraints}}

We will shortly discuss the projected gradient method to numerically solve \eqref{eq:optprob} and refer the reader e.g. to the work of Bertsekas \cite{DPB15} for a more detailed description of the method (in the discrete time setting).

Let $\cB = \{u\in\cX: a_i\le u_i \le b_i, i=1,\dots,m\},\ m\le n,$ denote a box. We can define a general projection as $\mathcal{P}_{\cB}:\R^n \rightarrow \cB$ mapping the space of the ensemble to specific box. Since we will fix the box $\cB$ for the following part, we will write $\cP$ instead of $\cP_{\cB}$. We define the projection $\mathcal{P}$ componentwise as 
\begin{align*}
(\cP(u))_i&=
\begin{cases}
a_i, & \textrm{if} \ \ u_i < a_i, \\
u_i, & \textrm{if} \ \ u_i \in [a_i,b_i] \\
b_i, & \textrm{if} \ \ u_i >b_i,
\end{cases}, \qquad i=1,\dots,m,\\
(\cP(u))_i&= u_i, \qquad i=m+1,\dots,n.
\end{align*}

The projected gradient method with step size $\alpha_k>0$ is based on the iteration 
\begin{equation}
\label{eq:pgm}
u^{k+1}(\alpha_k) = \cP(u^k-\alpha_k \nabla \Phi(u^k)).
\end{equation}

One can derive a continuous time by considering $\alpha_k$ going to zero. By using directional derivatives we obtain

\begin{equation}\label{eq:cont_projected_gradient}
\begin{split}
\left(\frac{du}{dt}\right)_i &= \begin{cases}
-\nabla_i \Phi(u),  &u_i\in(a_i,b_i),\\
-\nabla_i\Phi(u) \mathds{1}_{[0,\infty)}(-\nabla_i \Phi(u)), &u_i=a_i,\\
-\nabla_i\Phi(u) \mathds{1}_{(-\infty,0]}(-\nabla_i \Phi(u)), &u_i=b_i,\\
\end{cases}\qquad i=1,\dots,m,\\
\left(\frac{du}{dt}\right)_i &= -\nabla_i \Phi(u), \qquad i=m+1,\dots,n.
\end{split}
\end{equation}
More details can be found for the continuous time limit of the projected EKI in Section \ref{subsec:cont_time_limit}.

\begin{remark}\label{rem:smoothedsystem}
Since the right hand side (RHS) of \eqref{eq:cont_projected_gradient} is discontinuous, it is not obvious that a solution to this system exists. {To ensure unique existence we consider a smoothed version of \eqref{eq:cont_projected_gradient} by approximating the limit by ideas inspired from barrier methods. We introduce the parametrized, convex optimization problems}
\begin{equation}\label{eq:optprobsmoothed1}
\min\limits_{u\in\Omega} \Phi(u)-\frac{1}{\iota}\sum_{i=1}^{2m}log(-h_i(u)).
\end{equation}
{with parameter $\iota>0$ and inequality constraints $h_i(u)=a_i-u_i, \ i=1,\ldots,m$ and $h_{i+m}(u)=u_i-b_i, \ i=1,\ldots,m$.  As $\iota \to \infty$, the log barrier functions become closer to the indicator function of the feasible set of the original problem. We equivalently consider the problems
\begin{equation}\label{eq:optprobsmoothed2}
\min\limits_{u\in\Omega} \iota\Phi(u)-\sum_{i=1}^{2m}log(-h_i(u)).
\end{equation}
in the following, where we define $\tilde \Phi(u)=\iota \Phi(u)-\sum_{i=1}^{2m}log(-h_i(u))$.
We will approximate \eqref{eq:cont_projected_gradient} for $i\in\{1,\dots,m\}$ by
\begin{equation}\label{eq:smoothsystem_projected_gradient}
\frac{du}{dt} = -\iota \nabla \Phi+\sum_{i=1}^{2m}\frac{1}{h_i(u)}\nabla h_i(u)=-\nabla \tilde \Phi(u)\,.
\end{equation}
{For our theoretical results we will always consider the smoothed initial value problem.}
}

\end{remark}

\begin{thm}\label{thm:pgf}
Let $u_0\in\Omega$ and $u(t)$ denote the solution of the {smoothed initial value problem \eqref{eq:smoothsystem_projected_gradient}} with $u(0)=u_0$.

Further assume that $A^\top \Gamma^{-1}A$ is positive definite, and there exists a (unique global) minimizer $u^*_\iota$ of \eqref{eq:optprobsmoothed2}. Then for each $\iota>0$ it holds true that
$$ \lim\limits_{t\to\infty} u(t)=u^*_\iota,$$
i.e. the solution $u(t)$ converges to the (global) minimizer of \eqref{eq:optprobsmoothed2}.
\end{thm}

\begin{proof}
{
We define $V(u) = \frac12 \|u-u^*_\iota\|^2$ and prove that $V$ is a strict Lyapunov-function by the strict convexity of the optimization problem. The flow of $V$ satisfies

\begin{equation*}
\frac{dV(u)}{dt}  = \langle \frac{du}{dt},u-u^*_\iota\rangle =\langle -\nabla\tilde \Phi(u),u-u^*_\iota\rangle<0\,,
\end{equation*}
thus, the claim follows.
}
\end{proof}
\begin{remark}
By duality arguments (see \cite[11.2]{BV04} for more details), the accuracy of the approximation can be bounded by
\[
\Phi(u^*_\iota)-\Phi(u^*) \le \frac{2m}{\iota}\,,
\]
where $u^*$ denotes the minimizer of the original problem \eqref{eq:optprob}. In particular, $\Phi(u^*_\iota)\to \Phi(u^*) $ for $\iota\to \infty$ and thus $u^*_\iota\to u^*$.
\end{remark}

\begin{cor}\label{cor:pgf}
Let {$u_0\in\Omega$} and $u(t)$ denote the solution of the {smoothed initial value problem \eqref{eq:smoothsystem_projected_gradient}} with $u(0)=u_0$.

Further assume that there exists a (global) minimizer of  \eqref{eq:optprobsmoothed2}. Then it holds true that
{$$ \lim\limits_{t\to\infty} \Phi(u(t))=\Phi(u^*_\iota),$$}
where $u^*_{\iota}$ is a KKT-point of \eqref{eq:optprobsmoothed2}.
\end{cor}

\begin{proof}
Let $u^*_\iota$ be an arbitrary KKT-point of \eqref{eq:optprobsmoothed2}. The flow in the observation space is given by
{\[
\frac{dAu}{dt} = -\iota A\nabla \Phi+\sum_{i=1}^{2m}\frac{1}{h_i(u)}A\nabla h_i(u)\,,
\]}
which corresponds to the gradient flow of a strictly convex optimization problem in the observation space. Thus, by the same arguments as before in Theorem \ref{thm:pgf}, the claim follows.
\end{proof}

\subsubsection{{The preconditioned projected gradient method}}
We consider the preconditioned version of the iteration \eqref{eq:pgm} {in discrete time} given by
\begin{equation}
\label{eq:ppgm}
 u^{k+1}(\alpha) = \cP(u^k-\alpha D_k\nabla \Phi(u^k)),
 \end{equation}
where $D_k$ is a symmetric, positive definite matrix. It is well known that arbitrary choice of $D_k$ gives no descent for any choice of $\alpha>0$. {We will briefly discuss an example demonstrating that the preconditioning of the gradient flow does not lead to a descent direction in general (cp. \cite{DPB82}). This is highlighted in Figure \ref{fig:example}.}

\begin{figure}[h!]
\centering
 \includegraphics[scale=0.35]{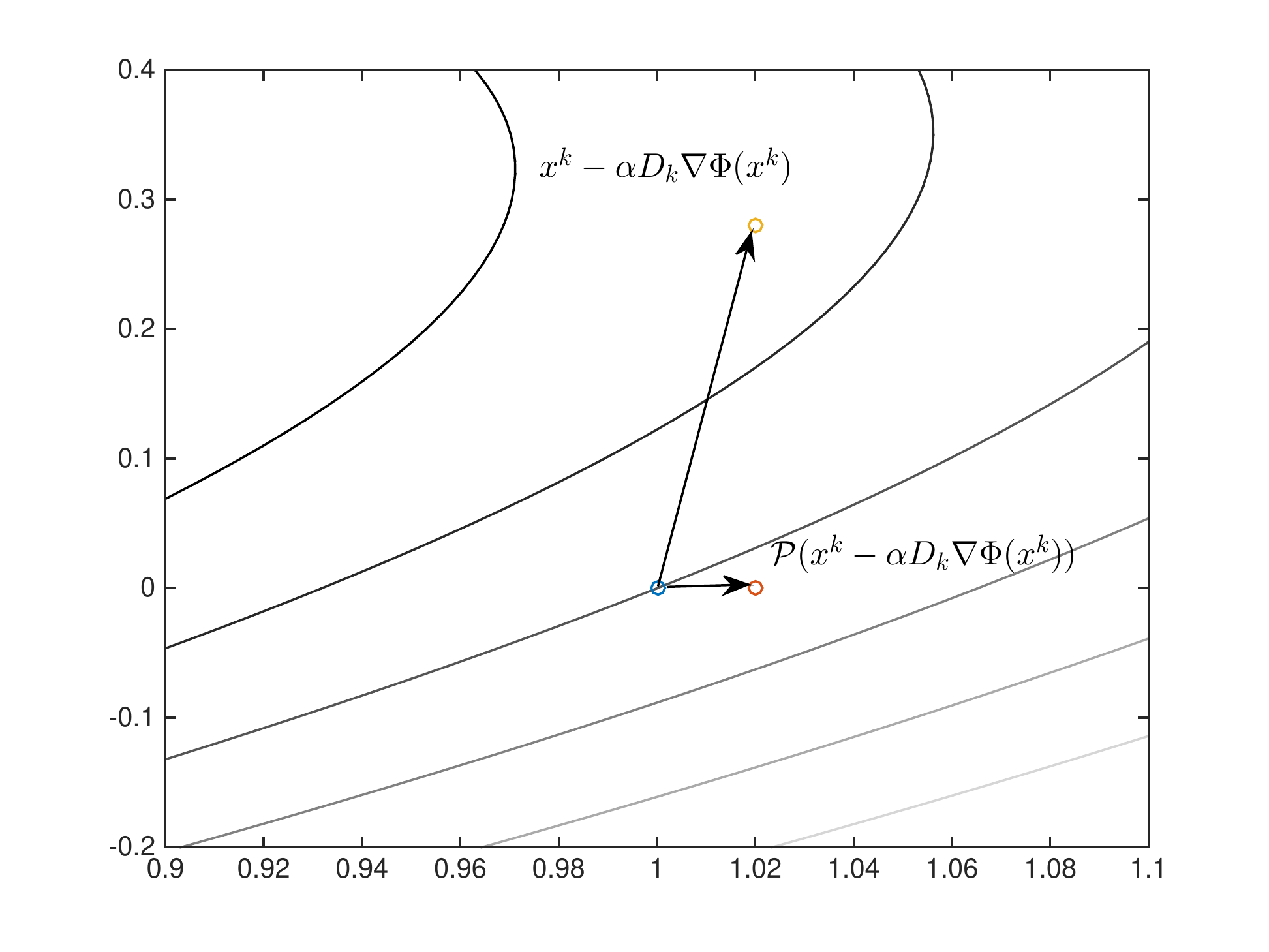}
\caption{Varying contour lines of the function $\Phi(x)$ defined in Example \ref{ex:contr}, with both the preconditioned descent direction in the unconstrained case and the projected preconditioned descent direction.}
 \label{fig:example}
\end{figure}

\begin{example}
\label{ex:contr}
We consider the 2-dimensional quadratic example: We define the quadratic function $$\Phi(x) = x^{\top} Qx+x_1 = (x_1-x_2)^2+2x_2^2+x_1, \quad \text{with}\quad Q = \begin{pmatrix} \phantom{-}1 & -1\\ -1 & \phantom{-}3 \end{pmatrix},$$
and consider the minimization problem of $F$ with the constraints $x_1\in\R,\ x_2\le0$. The gradient of $\Phi$ is given by $$\nabla \Phi(x) = \begin{pmatrix} 2x_1-2x_2+1\\  -2x_1+6x_2 \end{pmatrix}.$$

The current iterate is given by $x^k=\begin{pmatrix} 1\\ 0\end{pmatrix}$. 
 As a preconditioner, we consider the symmetric, positive definite symmetric matrix $$ D_k = \begin{pmatrix} 1 &2 \\ 2 &10\end{pmatrix}\,.$$ For $\alpha>0$, the next iteration is given by
$$ x^{k+1} = \cP(x^k - \alpha D_k\nabla \Phi(x^k))
= \begin{pmatrix} 1+\alpha\\ 0\end{pmatrix},$$
where $\cP$ is the projection onto $\R\times\R_{\le0}$. Then, $$\Phi(x^{k+1}) = (1+\alpha)^2+1+\alpha > 2 = \Phi(x^k)\,,$$ i.e. for all $\alpha>0$ the function value of the objective function increases.  
\end{example}

{
Example \ref{ex:contr} shows that a simple projection strategy for the EKI, which is a preconditioned gradient flow in the linear case, does not lead to a convergent descent method in general. 
 }

{To ensure a descent direction for the preconditioned projected gradient method, we follow the approach of \cite[Proposition 1]{DPB82}, which suggests to use matrices $D_k$ diagonal with respect to the subset of indices containing
\begin{equation}
\label{eq:indexbertsekas}
\mathcal{I}^+(u^k) = \bigg\{i\in\{1,\dots,m\}\mid u^k_i =a_i, \frac{\partial \Phi(u^k)}{\partial u^i} >0 \vee u^k_i = b_i, \frac{\partial \Phi(u^k)}{\partial u^i} <0\bigg\}.
\end{equation}
We will give more details on the modification of the preconditioner in the context of the EKI in the following. In particular, we will make use of data assimilation methods such as variance inflation to ensure a descent. This will be discussed in greater detail in Section \ref{sec:conv}.
}

\section{Convergence analysis}
\label{sec:conv}
{In this section we introduce a variant of the projected EKI and derive the continuous limit of the algorithm. We will provide a complete convergence analysis of the proposed modification in the linear setting. This will include an accuracy result of the proposed algorithm, the analysis of the ensemble collapse and the convergence to the truth.}

Recall that the componentwise projection onto the box $\cB$ is denoted by $\cP$. To incorporate the projection into our scheme we modify the procedure described in Section \ref{sec:bco}. We now define our \textit{prediction step} in its projected form as
\begin{equation*}
u_{n,\cP}^{(j)}=\cP(u_n^{(j)}),\quad \bar{u}_{\mathcal{P}} = \frac{1}{J} \sum_{j=1}^{J}(u^{(j)}_{n,{\mathcal{P}}}), \quad \bar{\mathcal{G}}_{\cP} = \frac{1}{J} \sum_{j=1}^{J}\mathcal{G}(u^{(j)}_{n,{\mathcal{P}}}), 
\end{equation*}
and similarly for the covariances 
\begin{align*}
C^{up}_{n,\cP} &= \frac{1}{J} \sum_{j=1}^{J}(u^{(j)}_{n,\cP} - \bar{u}_{\cP})\otimes (\mathcal{G}(u^{(j)}_{n,\cP}) - \bar{\mathcal{G}}_{\cP}), \\
C^{pp}_{n,\cP} &= \frac{1}{J} \sum_{j=1}^{J}(\mathcal{G}(u^{(j)}_{n,\cP}) - \bar{u}_{\cP})\otimes (\mathcal{G}(u^{(j)}_{n,\cP}) - \bar{\mathcal{G}}_{\cP}).
\end{align*}

{
Then by using these estimates we can construct our update formula in its closed form which is
\begin{equation}
\label{eq:up_proj}
\begin{cases}
\tilde u^{(j)}_{n+1,\mathcal{P}} = u^{(j)}_{n,\mathcal{P}} + C^{up}_{n,\mathcal{P}}(C^{pp}_{n,\mathcal{P}} + h^{-1} \Gamma)^{-1}(y - \mathcal{G}(u^{(j)}_{n.\mathcal{P}})),\\
u^{(j)}_{n+1,\mathcal{P}} = \cP(\tilde u^{(j)}_{n+1,\mathcal{P}}).
\end{cases}
\end{equation}}

\subsection{Continuous time limit}\label{subsec:cont_time_limit}

{We now derive the continuous time limit for the projected EnKF for inverse problems. We begin be recalling the equations in closed form, given by the component-wise increments}

\begin{align*}
(u_{n+1}^{(j)})_i-(u_n^{(j)})_i &=\cP\Big([\cP(\tilde u_n^{(j)})]_i+\Big[C_{n,\cP}^{up}(C_{n,\cP}^{pp}+\frac1h\Gamma)^{-1}(y-\cG(\cP(u_{n}^{(j)}))\Big]_i\Big) \\
&-\cP\Big([\cP(\tilde u_{n}^{(j)})]_i\Big).
\end{align*}
{By using the Neumann expansion for part of the Kalman gain, we observe for positive-semidefinite {$C\in \mathbb R^{K\times K}$} that
\begin{align}
\bigg(\frac1h\Gamma+C\bigg)^{-1} 
\label{eq:series}
&=h\Gamma^{-1}+\sum\limits_{k=1}^\infty h^{k+1}(\Gamma^{-1}C)^k\Gamma^{-1}.
\end{align}
By defining $v_i:=\bigg[C_{n,\cP}^{up}\Gamma^{-1}(y-\cG(\cP(u_{n}^{(j)}))\bigg]_i$, using \eqref{eq:series}
and by using the definition of directional derivatives, we obtain}

\begin{equation}
\label{eq:dd}
\begin{split}
\lim\limits_{h\to0} \frac{(u_{n+1}^{(j)})_i-(u_n^{(j)})_i}h &= 
\begin{cases}
v_i, 	&(\tilde u_n^{(j)})_i\in(a_i,b_i),\\
\mathds{1}_{[0,\infty)}(v_i)v_i, 	&(\tilde u_n^{(j)})_i\le a_i,\\
\mathds{1}_{(-\infty,0]}(v_i)v_i, 	&(\tilde u_n^{(j)})_i\ge b_i,
\end{cases}\quad i=1,\dots,m\\
\lim\limits_{h\to0} \frac{(u_{n+1}^{(j)})_i-(u_n^{(j)})_i}h &=	v_i(u_t^{(j)}), \quad i=m+1,\dots,n.	
\end{split}
\end{equation}
\\
Finally, we obtain the continuous time limit for the EnKF by

\begin{equation}\label{eq:projection_cont_time}
\begin{split}
\left(\frac{du^{(j)}}{dt}\right)_i &= \begin{cases} v_i(u_t^{(j)}), & (u_t^{(j)})_i\in(a_i,b_i),\\
																\mathds{1}_{[0,\infty)}(v_i(u_t^{(j)}))v_i(u_t^{(j)}), & (u_t^{(j)})_i = a_i,\\
																\mathds{1}_{(-\infty,0]}(v_i(u_t^{(j)}))v_i(u_t^{(j)}), & (u_t^{(j)})_i = b_i,
																\end{cases}\quad i=1,\dots,m\\
\left(\frac{du^{(j)}}{dt}\right)_i &=	v_i(u_t^{(j)}) \quad i=m+1,\dots,n,		
\end{split}				
\end{equation}
where $v_i(u^{(j)})$ is given by $$v_i(u^{(j)}) = \left[C^{up}(u)\Gamma^{-1}(y- \cG(u^{(j)}))\right]_i.$$

In the linear case we can write $v(u^{(j)})$ as

\begin{equation}\label{eq:precond}
v(u^{(j)}) = -C(u)\nabla \Phi(u^{(j)}),
\end{equation}
with $\Phi(u) =\frac12 \| Au-y\|_{\Gamma}^2$ from the minimization problem \eqref{eq:optprob}.

{
\begin{remark}
{Due to the projection onto the feasible set, the RHS of \eqref{eq:projection_cont_time} is discontinuous and {for our theoretical analysis} we will consider the smoothed system }

{
\begin{equation}\label{eq:smoothedsystem_projected_EKI}
\frac{du^{(j)}}{dt} = -\iota C(u)\nabla \Phi(u^{(j)})+\sum_{i=1}^{2m}\frac{1}{h_i(u)}\nabla h_i(u),
\end{equation}
for $j=1,\dots,J$. 
}
\end{remark}
}

\subsection{Variance inflation}

The EnKF and the ESRF are known to have certain difficulties in a high dimensional setting. One of the cases is where the dimension of the parameter space is considerably larger than the dimension of the ensemble. As a result this can undervalue the importance of the covariance matrices. One common way to alleviate this issue is through variance inflation \cite{JLA07,JLA09,TMK16}. 

The idea is to inflate the empirical covariance by another covariance matrix which is of the prior. {This is usually specified through additive variance inflation defined below as
\begin{equation}
\label{eq:add}
C(u) \rightarrow \vartheta C_0 + C(u),
\end{equation}}
where $C_0$ is our prior covariance and $\vartheta>0$ is a constant. If we relate this to EKI, this was first applied in \cite{SS17}. By incorporating this form of inflation our gradient flow structure is modified to
\begin{equation}
\label{eq:inf_limit}
\frac{du^{(j)}}{dt} =  - (\vartheta C_0 + C(u)) \nabla_u \Phi(u^{(j)};y).
\end{equation}
We note if one was to apply it for the projection the limit would differ in that the $u$ would be projection under $\mathcal{P}$. 

\subsection{Convergence analysis in the linear case}
\label{sec:grad_flow}
In the linear case we view the continuous time limit of projected EnKF as preconditioned gradient flow.

\subsubsection{Transformed method for the EnKF}
In Example \ref{ex:contr} we have seen that in the case of preconditioned projected gradient methods, it is not possible to ensure a descent direction for an arbitrary choice of a positive definite matrix $D_k$. Based on the results of Bertsekas \cite{DPB82}, we will transform \eqref{eq:projection_cont_time} in a way such that the preconditioner is diagonal with respect to an index set $\cI^+(u)$ which is built similar to the set from \eqref{eq:indexbertsekas}. Since we consider a system of particles we will use a preconditioner which is diagonal with respect to an index set which depends on the whole ensemble of particles $u = (u^{(j)})_{j=1,\dots,J}$, in particular we set 

$$\cI^+(u):= \bigg\{i\in\{1,\dots,m\}| \bar{u_i}=a_i, \frac{\partial \Phi(\bar u)}{\partial x^i}>0 \quad \vee \quad \bar{u_i} = b_i, \frac{\partial \Phi(\bar u)}{\partial x^i}<0\bigg\}.$$

Similarly, we could also choose the index set to be 
\begin{align*}\hat\cI^+(u):= \underset{j\in\{1,\dots,J\}}{\cup}\bigg\{i\in\{1,\dots,m\} \mid &u_i^{(j)}=a_i, \frac{\partial \Phi(u^{(j)})}{\partial x^i}>0 \quad \vee \\ \quad &u_i^{(j)} = b_i, \frac{\partial \Phi(u^{(j)})}{\partial x^i}<0\bigg\}.
\end{align*}

In this work we will focus on the choice of $\cI^+(u)$. Therefore, we consider the preconditioned gradient flow given by 

\begin{equation}\label{eq:transformedEnKF}
\begin{split}
\left(\frac{du^{(j)}}{dt}\right)_i &= \begin{cases}
											p_i(u_t^{(j)}), &(u_t^{(j)})_i \in(a_i,b_i)\\
											\mathds{1}_{[0,\infty)}(p_i(u_t^{(j)}))(p_i(u_t^{(j)})), & (u_t^{(j)})_i=a_i\\
											\mathds{1}_{(-\infty,0]}(p_i(u_t^{(j)}))(p_i(u_t^{(j)})), & (u_t^{(j)})_i=b_i,
											\end{cases},\quad i=1,\dots,m,\\
\left(\frac{du^{(j)}}{dt}\right)_i &= p_i(u_t^{(j)}),\quad i=m+1,\dots,n,
\end{split}
\end{equation}
where $p(u_t^{(j)}) = -D(u_t)\nabla \Phi(u_t^{(j)})$ and the preconditioner is given by 

\begin{equation}\label{eq:cont_precondtioner}
(D(u_t))_{i,j} = \begin{cases}
						(C(u_t)+\varepsilon I)_{i,j}, & i,j\in\{1,\dots,n\}\setminus \cI^+(u_t)\\
						0,	& i\in\cI^+(u_t), j\neq i \vee j\in\cI^+(u_t), j\neq i\\
						\varepsilon, 	&i=j\in\cI^+(u_t).
						\end{cases} 
\end{equation}
Let $\{x^{(j)}\}_{j=1}^{J}$ be a system of particles in $\R^n$. Without loss of generality we assume $\cI^+(x)=\{1,\dots,r\}$ for $r\le m$. Hence, the the preconditioner $D(x)$ can be written as
\begin{equation*}
\begin{pmatrix}
\varepsilon I_r & 0\\
0 & \hat C(x)+\varepsilon I_{n-r}
\end{pmatrix} = 
\begin{pmatrix}
0 & 0\\
0 & \hat C(x)
\end{pmatrix} + 
\varepsilon I_n,
\end{equation*}
where $(\hat C(x))_{i,j\in\{1,\dots,n-r\}} = (C(x))_{i,j\in\{r+1,\dots,n\}}$. 

{
\begin{remark}
\label{rem:vi}
Consider the continuous time limit of the projected EnKF in equation \eqref{eq:projection_cont_time}. When we introduce variance inflation for \eqref{eq:precond} with the help of a diagonal matrix, e.g. by
\begin{equation}\label{eq:projectedEnKFVI}
v(u^{(j)}) = -(C(u)+\varepsilon I)\nabla\Phi(u^{(j)}),
\end{equation}
then the preconditioner $D(u) = C(u)+\varepsilon I$ is diagonal with respect to the index set $\cI^+(u)$ and can be written in the form of \eqref{eq:cont_precondtioner}.

{
This can be seen as follows: since $u^{(j)}$ evolves by \eqref{eq:transformedEnKF} we have that $u_i^{(j)}\in[a_i,b_i]$ for all $i$ and $j$. {For simplicity, we assume $m=n$.} Let $i\in\cI^+(u)$, then it follows 
\begin{equation*}
\bar u_i = a_i \qquad \vee\qquad \bar u_i = b_i.
\end{equation*}
Since the particles are feasible for all $t \geq 0$, we obtain 
\begin{equation*}
u_i^{(j)}=a_i \quad \forall j \qquad \vee \qquad u_i^{(j)} = b_i \quad \forall j,
\end{equation*}
and finally $u_i^{(j)}-\bar u_i = 0$ for all $j$. This leads to 
\begin{equation*}
(C(u)+\varepsilon I)_{ik} = \frac1J\sum\limits_{j=1}^J (u_i^{(j)}-\bar u_i)(u_k^{(j)}-\bar u_k) = 0,
\end{equation*}
for $i\in\cI^+(u)$, $k\in\{1,\dots,n\}$ and $i\neq k$.
}

 Hence, we can view the projected EnKF with variance inflation as special case of the presented method. 
\end{remark}
}

{
\subsubsection{Transformed method for the ESRF}

The continuous limit \eqref{eq:projection_cont_time} and the gradient flow \eqref{eq:precond} can be adapted to the case of ESRF. The same can be said for the projected preconditioned limit above. In order to do so the closed form would need to be changed such that the limit follows a similar form to \eqref{eq:limit1}.

The resulting transformed method for the ESRF is given by

\begin{equation}\label{eq:transformedESRF}
\begin{split}
\left(\frac{du^{(j)}}{dt}\right)_i &= \begin{cases}
											p_i(u_t^{(j)}), &(u_t^{(j)})_i \in(a_i,b_i),\\
											\mathds{1}_{[0,\infty)}(p_i(u_t^{(j)}))(p_i(u_t^{(j)})), & (u_t^{(j)})_i=a_i,\\
											\mathds{1}_{(-\infty,0]}(p_i(u_t^{(j)}))(p_i(u_t^{(j)})), & (u_t^{(j)})_i=b_i,
											\end{cases}\quad i=1,\dots,m,\\
\left(\frac{du^{(j)}}{dt}\right)_i &= p_i(u_t^{(j)}),\quad i=1,\dots,m,
\end{split}
\end{equation}
where now $p(u_t^{(j)}) = D(u_t)(\frac12\nabla \Phi(u_t^{(j)})+\frac12 \nabla\Phi(\bar u_t))$ and the preconditioner is again given by \eqref{eq:cont_precondtioner}.

For simplicity we will focus in the following on the transformed method for EKI \eqref{eq:transformedEnKF}. All results can be easily adapted to the transformed method \eqref{eq:transformedESRF} for the ESRF.
}

\subsubsection{{Convergence Results}}
{Let $u_{\iota}^*$ be KKT-point for \eqref{eq:optprobsmoothed2} and $(u_t^{(j)})$ be the solution of the smoothed system of \eqref{eq:transformedEnKF}
\begin{equation}
\label{eq:transformedEnKF_smooth}
\frac{du_t^{(j)}}{dt} = -\iota D(u_t)\nabla\Phi(u_t^{(j)})+\sum\limits_{i=1}^{2m}\frac1{h_i(u)}\nabla h_i(u),
\end{equation}}
where $h$ has been defined in Remark \ref{rem:smoothedsystem}. Our aim is to prove the convergence of the residual $\tilde r_t^{(j)} := u_t^{(j)} -  u_{\iota}^*$ for $t\to\infty.$ To show convergence we have to control the empirical covariance $C(u_t)$ over time. In particular, in the following proposition we will prove that we can bound the ensemble spread $e_t^{(j)} = u_t^{(j)}-\bar{u}_t$ over time.

\begin{proposition}
\label{prop:ensemblespread}

Let $u_0^{(j)}\in\Omega,$ $j=1,\dots,J$ be the initial ensemble and $u(t)$ denotes the solution for the {smoothed} flow \eqref{eq:transformedEnKF_smooth}.
Then, it holds true that $$ \frac1J\sum\limits_{j=1}^J |e_t^{(j)}|^2 \le \frac1J\sum\limits_{j=1}^J |e_0^{(j)}|^2,$$ for all $t\ge0$.
\end{proposition}
\begin{proof}
We consider $V(u) = \frac1J \sum\limits_{j=1}^J \frac12|u^{(j)}-\bar u|^2$ and show that $V(u)\le 0$, which implies the monotonicity of the quantity $e$. We have
\begin{align*}
\frac{dV(u_t)}{dt} =& -\frac1J\sum\limits_{j=1}^J \langle u_t^{(j)}-\bar u_t, {\iota}D(u_t)A^\top \Gamma^{-1}A(u_t^{(j)} -\bar u_t)\rangle\\
&-\frac1J\sum\limits_{j=1}^J \langle u_t^{(j)}-\bar u_t, \nabla \tilde h(u_t^{(j)}) -\bar  \nabla \tilde h(u_t))\rangle,
\end{align*}
with $\tilde h(u)=\frac{1}{\iota}\sum_{i=1}^{2m}\log(-h_i(u))$. The first inner product can be straightforwardly shown to be smaller or equal than zero by exploiting the linearity of $\nabla \Phi$. The convexity of $\tilde h$ yields   
\[
-\frac1J\sum\limits_{j=1}^J \langle u_t^{(j)}-\bar u_t, \nabla \tilde h(u_t^{(j)}) -\bar  \nabla \tilde h(u_t))\rangle\le 0\,,
\]
which proves the monotonicity of the quantity $e$.

\end{proof}

We are now ready to prove the following theorem regarding the residual in the parameter space {{$\tilde{r}^{(j)}_t=u_t^{(j)}-u^*_\iota$}, where $u^*_\iota$ is a KKT-point for \eqref{eq:optprobsmoothed2}.}
\begin{thm}\label{thm:residuals}
{{Let $u_0^{(j)}\in\Omega$}, $j=1,\dots,J$ be the initial ensemble and {$u(t)$ denotes the solution of \eqref{eq:transformedEnKF_smooth}}, {$A^\top \Gamma^{-1}A$ be positive definite}, $\varepsilon>0$, {$\Omega\neq\emptyset$} and {$u^*_\iota$} be the KKT-point for \eqref{eq:optprobsmoothed2}. Then it holds true that {$$\lim\limits_{t\to\infty}\frac1J\sum\limits_{j=1}^J |\tilde r_t^{(j)}|^2 = \lim\limits_{t\to\infty}\frac1J\sum\limits_{j=1}^J |u_t^{(j)}-u^*_\iota|^2= 0.$$}}
\end{thm}

\begin{proof}
{By assumption, we have that $A^\top \Gamma^{-1}A$ is positive definite and $\Omega\neq\emptyset$, i.e. there exists a unique global minimizer ${u^*_\iota}$ of \eqref{eq:optprobsmoothed2}.} 

We define the function $V(u) = \frac1J\sum\limits_{j=1}^J \frac12|u^{(j)}-{u^*_\iota}|^2$ and will prove $$\frac{dV(u_t)}{dt}<0,$$ for $u_t= \{u_t^{(j)}\}_{j=1}^J$.

The variance inflation breaks the subspace property of the EKI and yields the convergence to the KKT point:

\begin{align*}
\frac1J\sum\limits_{j=1}^J \frac{d \frac12 |u_t^{(j)} - {u^*_\iota}|^2}{dt} = &\frac1J\sum\limits_{j=1}^J \langle u_t^{(j)} - {u^*_\iota}, -D(u)\nabla \Phi(u_t^{(j)})-\nabla \tilde h(u_t^{(j)})\rangle\\
=&-\frac1J\sum\limits_{j=1}^J \langle u_t^{(j)} - {u^*_\iota}, C(u)A^\top \Gamma^{-1}A(u_t^{(j)} - {u^*_\iota})\rangle\\
&-\frac1J\sum\limits_{j=1}^J\epsilon  \langle u_t^{(j)} - {u^*_\iota},\nabla \Phi(u_t^{(j)})+\nabla \tilde h(u_t^{(j)})\rangle\\
=&-\frac{1}{J^2}\sum\limits_{j,k=1}^J (u_t^{(j)} - {u^*_\iota})^\top(u_t^{(k)} - {u^*_\iota})A^\top \Gamma^{-1} A (u_t^{(k)} - {u^*_\iota})^\top(u_t^{(j)} - {u^*_\iota})\\
&-\frac1J\sum\limits_{j=1}^J\epsilon  \langle u_t^{(j)} - {u^*_\iota},\nabla \Phi(u_t^{(j)})+\nabla \tilde h(u_t^{(j)})\rangle\\
<& \ 0\,.
\end{align*}
\end{proof}

\begin{cor}\label{cor:residuals}
{{Let $u_0^{(j)}\in\Omega$}, $j=1,\dots,J$ be the initial ensemble, $\varepsilon>0$ and assume that there exists a (global) minimizer for \eqref{eq:optprobsmoothed2}. Then it holds true that $$\lim\limits_{t\to\infty}\frac1J\sum\limits_{j=1}^J  |\Phi(u_t^{(j)})-\Phi({u^*_\iota})|^2 = 0,$$
where ${u^*_\iota}$ is a KKT-point of \eqref{eq:optprobsmoothed2}.}
\end{cor}
\begin{proof}
{The claim follows similarly to the proof of Corollary \ref{cor:pgf}.}
\end{proof}

Instead of using variance inflation with constant multiplier $\varepsilon$, we will reduce the variance inflation over time. In particular, we will set $\varepsilon(t) = \frac{1}{t^\alpha+R}$, with $\alpha\in(0,1)$ and $R>0$. We will quantify the ensemble collapse with given rate in the following Proposition.
{
\begin{proposition}
\label{prop:ensemblespread_rate}

Let $u_0^{(j)}\in\Omega,$ $j=1,\dots,J$ be the initial ensemble and {$u(t)$ denotes the solution of \eqref{eq:transformedEnKF_smooth}}, $A^\top \Gamma^{-1}A$ be positive definite and $\varepsilon(t) = \frac{1}{t^\alpha+R}$, with $\alpha\in(0,1)$ and $R>0$. Then it holds true that $$ \frac1J\sum\limits_{j=1}^J |e_t^{(j)}|^2 \in \cO(t^{-(1-\alpha)}).$$
\end{proposition}
\begin{proof}
We consider again $V(u) = \frac1J \sum\limits_{j=1}^J \frac12|u^{(j)}-\bar u|^2$ and use
\begin{align*}
\frac{dV(u_t)}{dt} =& -\frac1J\sum\limits_{j=1}^J \langle u_t^{(j)}-\bar u_t, D(u_t)A^\top \Gamma^{-1}A(u_t^{(j)} -\bar u_t)\rangle\\
&-\frac1J\sum\limits_{j=1}^J \langle u_t^{(j)}-\bar u_t, \nabla \tilde h(u_t^{(j)}) -\bar  \nabla \tilde h(u_t))\rangle\\
&\le-\frac1J\sum\limits_{j=1}^J \langle u_t^{(j)}-\bar u_t, \varepsilon(t)A^\top \Gamma^{-1}A(u_t^{(j)} -\bar u_t)\rangle,
\end{align*}
similarly to the proof of \ref{prop:ensemblespread}. Thus, it follows that
\begin{equation*}
\frac{dV(u_t)}{dt} \le -\varepsilon(t)\sigma_{\min}(A^\top\Gamma^{-1}A)V(u_t).
\end{equation*}
Finally, using the bound
\begin{align*}
V(u_0)\ge \int_0^t\sigma_{\min}(A^\top\Gamma^{-1}A)\varepsilon(s)\,ds V(u_t),
\end{align*}
for all $t\ge0$ implies $$ \frac1J\sum\limits_{j=1}^J |e_t^{(j)}|^2 \in \cO(t^{-(1-\alpha)}).$$

\end{proof}
}

As shown before we can also prove the convergence  of the residuals in the parameter space when we reduce the variance inflation over time.

{
\begin{cor}\label{cor:residuals2}
Let $u_0^{(j)}\in\Omega$, $j=1,\dots,J$ be the initial ensemble and {$u(t)$ denotes the solution of \eqref{eq:transformedEnKF_smooth}}, {$A^\top \Gamma^{-1}A$ be positive definite} and $\varepsilon(t) = \frac{1}{t^\alpha+R}$, with $\alpha\in(\frac12,1)$ and $R>0$. Furthermore, let $u_\iota^*$ be the {KKT-point of  \eqref{eq:optprobsmoothed2}}. Then it holds true that $$\frac1J\sum\limits_{j=1}^J |\tilde r_t^{(j)}|^2 \in \cO(t^{-(1-\alpha)}).$$
\end{cor}
\begin{proof}
We define the function $V(u) = \frac1J\sum\limits_{j=1}^J \frac12|u^{(j)}-u_\iota^*|^2$ and will prove $$\frac{dV(u_t)}{dt}<0,$$ for $u_t= \{u_t^{(j)}\}_{j=1}^J$. Similarly to the proof of Theorem \ref{thm:residuals} we obtain
\begin{align*}
\frac{d V(u_t)}{dt} =&-\frac{1}{J^2}\sum\limits_{j,k=1}^J\langle u_t^{(j)}-{u_\iota^*},u_t^{(k)}-{u_\iota^*}\rangle\langle A^\top \Gamma^{-1}A(u_t^{(k)}-{u_\iota^*}),u_t^{(j)}-{u_\iota^*}\rangle\\
&-\frac1J\sum\limits_{j=1}^J(\varepsilon(t)  \langle u_t^{(j)} - {u_\iota^*},\nabla \Phi(u_t^{(j)})\rangle +  \langle u_t^{(j)} - {u_\iota^*},\nabla \tilde h(u_t^{(j)})\rangle).\\
\end{align*}
and we conclude with
$$\frac{dV(u_t)}{dt} \le -\varepsilon(t)\sigma_{\min}(A^\top\Gamma^{-1}A)V(u_t).$$
\end{proof}

\begin{remark}
As before the linear convergence analysis presented in this section can be generalized to the ESRF. To do so the Lyapunov function $V$ needs to be modified to include the term $\bar{u}_t$, based on its gradient flow \eqref{eq:limit2}.
\end{remark}

\section{Numerical experiments}
\label{sec:num}

For this section we introduce the numerical experiments necessary for the projection based optimization. We conduct our experiments for both a linear and non-linear setting. Our forward models will consist of elliptic PDEs. In each case the projection will be tested through its respective continuum limit \eqref{eq:projection_cont_time}. {Aside from aiming to show the recovery of truth, or underlying unknown denoted as $u^{\dagger}$}, we also show the spread of the {ensemble and also the difference to the (not necessarily unique) KKT-point $u^*$ of \eqref{eq:optprob}, which we have approximately computed with the MATLAB function $\mathtt{fmincon}$. These are summarized below.  \\
\begin{enumerate}
\item Error: $|e^{(j)}|^2_2 = |u_t^{(j)}-\bar u_t|^2_2$.
\item Residual: $|r^{(j)}|^2_2 =|u_t^{(j)}-u^{\dagger}|^2_2$.
\item KKT-Residual: $|\tilde r^{(j)}|^2_2 = |u_t^{(j)}-u^*|^2$.
\item Error evaluated through operator $A$: $|Ae|^2_{\Gamma}$.
\item Residual evaluated through operator $A$: $|Ar|^2_{\Gamma}$. 
\item Error in the cost function: $|\Phi(u^{(j)})-\Phi(u^*)|^2_2$.\\
\end{enumerate}
}
{We emphasize for the numerics that we consider a system with noise-free measurements to verify the presented theory and that we will consider a smoothened version of algorithms in the continuous time limits.}
\subsection{Linear inverse problem}
\label{ssec:linear}
We now seek to numerical check whether the BC optimization introduced in Section \ref{sec:conv} works in practice. For the continuum limit we solve the ODE through the MATLAB solver $\mathtt{ode45}$. In both cases we use an ensemble size of $J=5$.  Our forward model will be a linear 1D elliptic PDE of the form, where we seek a solution 
 $p \in  \mathcal{U}:= H^1_0(D)$ from 
\begin{align}
\label{eq:fwd1}
\frac{d^2p}{dx^2} + p &= u \ \ \ x \in D, \\
\label{eq:bc1}
p &= 0  \ \ \ x \in \partial D.
\end{align}
The inverse problem associated with \eqref{eq:fwd1} is the recovery of $u \in \mathcal{X}$ from 16 pointwise measurements of $p$. {Our forward solver for \eqref{eq:fwd1} is a piecewise finite element method with mesh size $h=2^{-8}$.} We will assume our computational domain is $D=(0,\pi)$. Our mapping $\mathcal{G}(\cdot) = A \cdot$  is assumed to be linear where we set $A = \mathcal{O} \circ \mathcal{G}$, where $\mathcal{O}:\mathcal{U} \rightarrow \R^K$ is the observational operator and $G:\mathcal{X} \rightarrow \mathcal{U}$ is the forward operator. We specify the covariance of the noise as $\Gamma = \gamma^2I$ where $\gamma = 0.01$, and we choose $T =10^6$. {The initial ensemble $u_0$ is chosen through a Fourier basis representation, and the inflation parameters are taken as $\alpha = 0.75$ and $R=1$.}

As stated in Remark \ref{rem:vi}, the projected EnKF with variance inflation can be viewed as a special case of the transformed version provided in \eqref{eq:transformedEnKF}. This suggests that both methods should perform similarly and outperform the original EnKF with no constraint. For the numerics we now specify the projected EnKF as the projected EnKF without variance inflation and the transformed EnKF as the projected EnKF with variance inflation. For the linear case we compared the projected EnKF, the transformed version and the original method without projecting onto the box. {Note that the numerical solution of both the projected EnKF and the transformed EnKF are based on a smoothed version of the indicator function $1_y(x)$ by a linear function $\tilde 1_y(x)$  with $\tilde 1_y(y)=1$ and $\tilde 1_y(y\pm\iota)=0$ for the upper and lower bound respectively. \smallskip

}
Our numerics will consist of two different cases:
\begin{enumerate}
\item The truth $u^\dagger$ lies outside the box and $\cO$ gives full observations.
\item The truth $u^\dagger$ lies outside the box, and $\cO$ gives low dimensional observations. \smallskip
\end{enumerate}
To understand the effect of the different forms of observations, if we have full observations then there exists a unique KKT-point to \eqref{eq:optprob}, implying the true parameter is a KKT point. If we have low dimensional observations then $A^\top\Gamma^{-1}A$ is only positives semi-definite. Thus, we can only expect convergence of the cost functions. For the first case we we choose full observations and for the latter we choose 15 observations.

\begin{figure}[!htb]
	\begin{subfigure}[c]{0.49\textwidth}
	\includegraphics[width=1.1\textwidth]{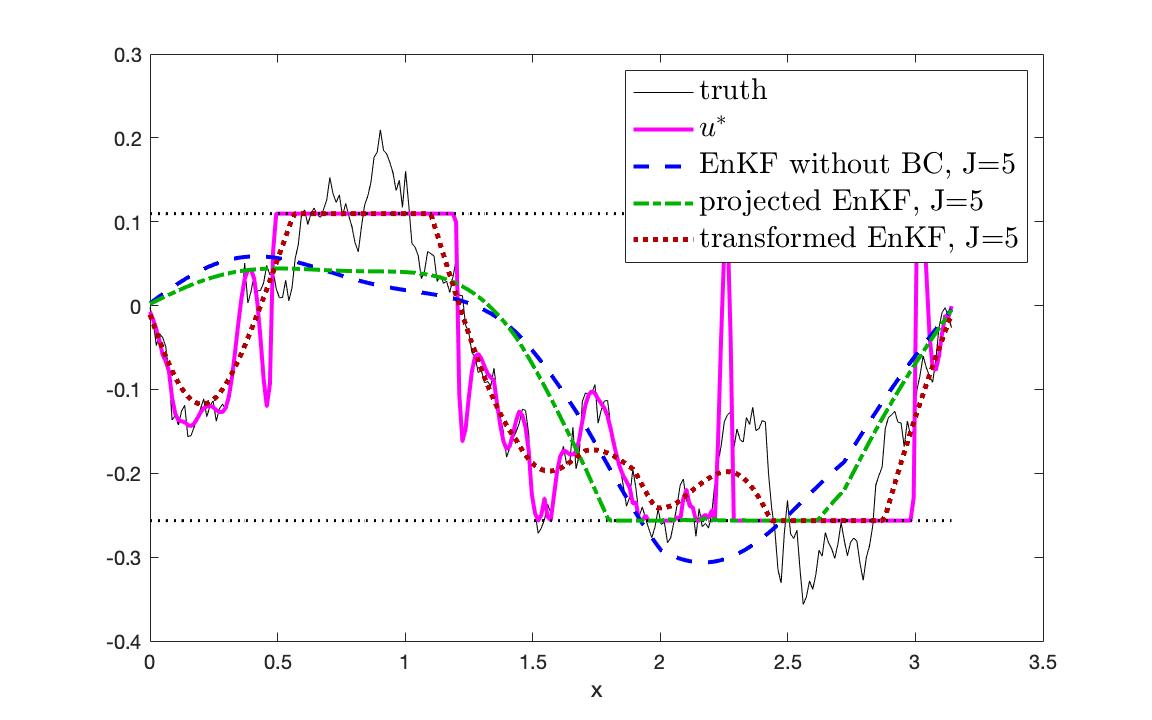}
	\end{subfigure}
	\begin{subfigure}[c]{0.49\textwidth}
	\includegraphics[width=1.1\textwidth]{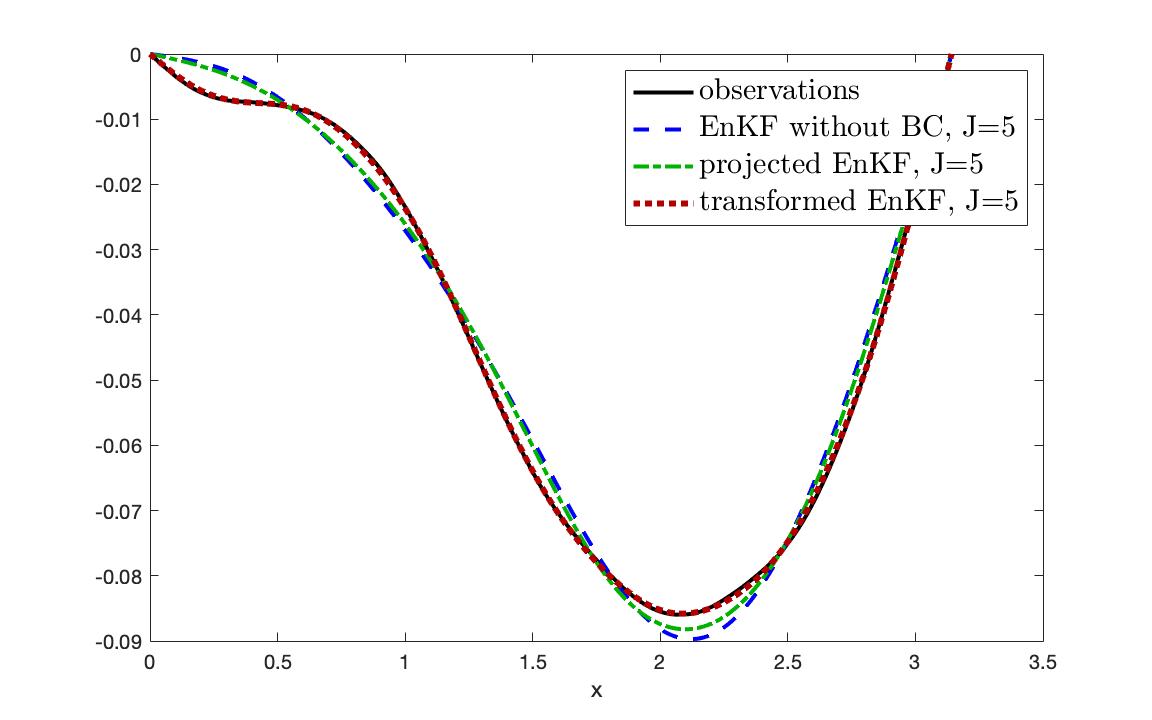}
	\end{subfigure}
    \caption{Transformed EnKF estimation in comparison to the EnKF estimation and the projected EnKF estimation. $J=5$ particles have been simulated.}\label{fig:A1A3_Test}
\end{figure} 

\begin{figure}[!htb]
	\begin{subfigure}[c]{0.49\textwidth}
	\includegraphics[width=1.1\textwidth]{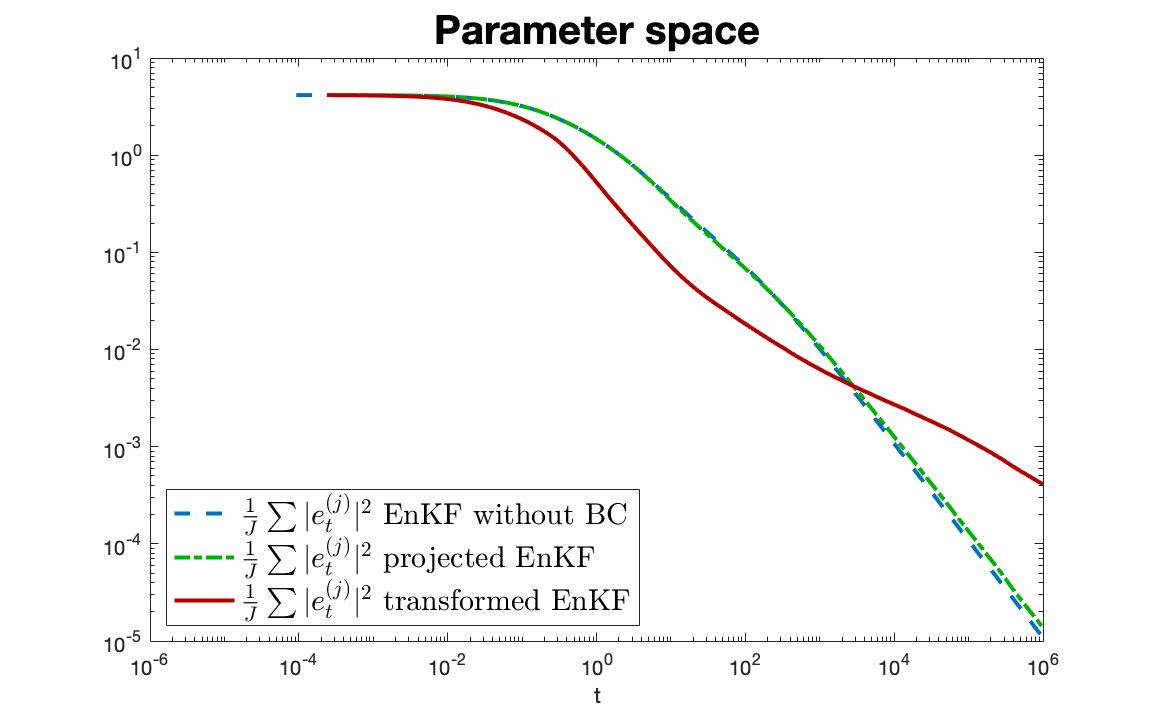}
	\end{subfigure}
	\begin{subfigure}[c]{0.49\textwidth}
	\includegraphics[width=1.1\textwidth]{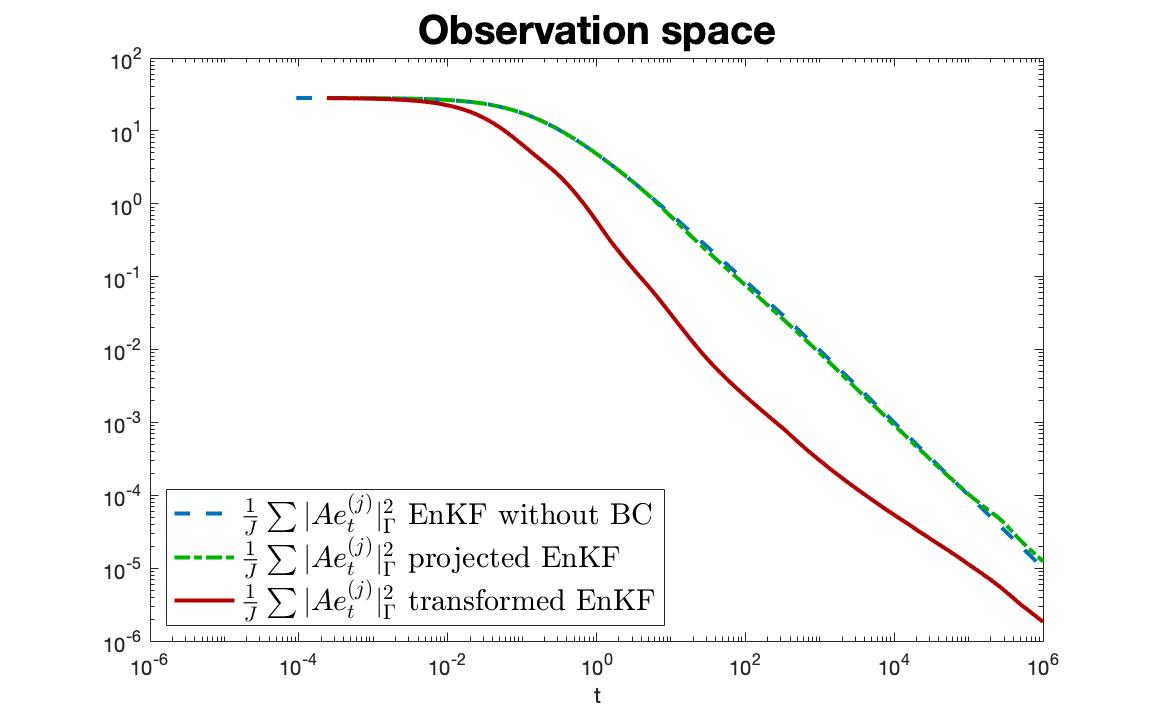}
	\end{subfigure}
    \caption{Ensemble spread in the transformed EnKF in comparison to the EnKF and the projected EnKF. $J=5$ particles have been simulated.}\label{fig:A1A3_Tspread}
\end{figure} 

\begin{figure}[!htb]
	\begin{subfigure}[c]{0.49\textwidth}
	\includegraphics[width=1.1\textwidth]{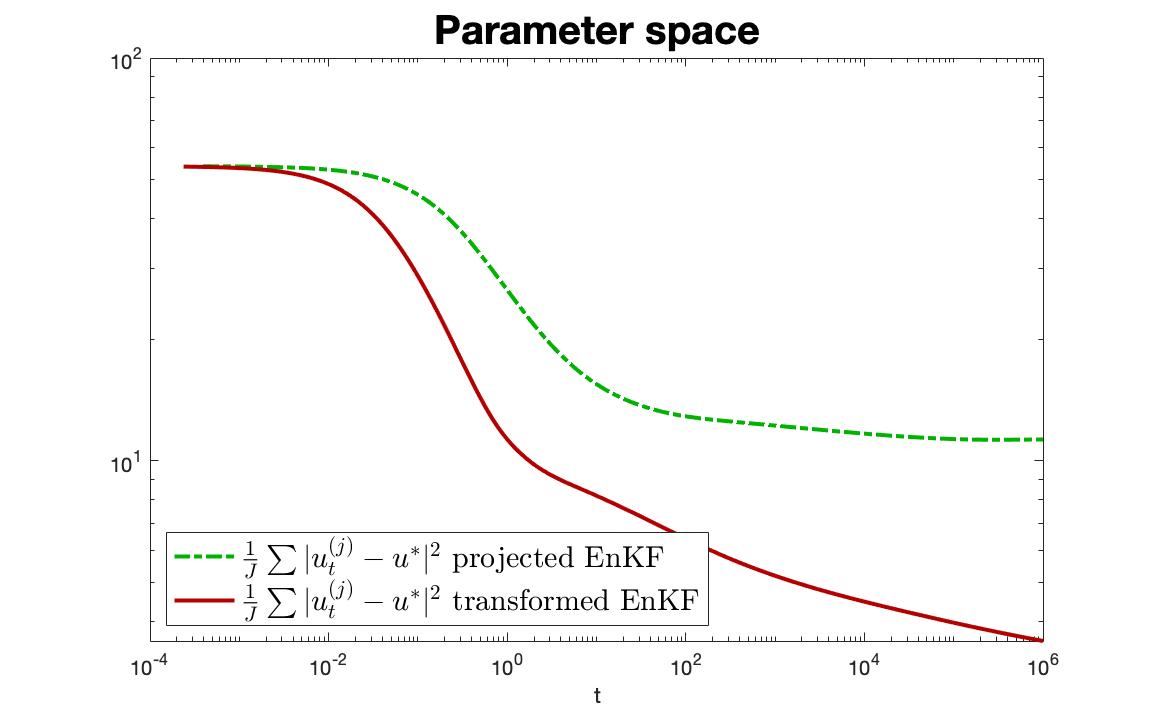}
	\end{subfigure}
	\begin{subfigure}[c]{0.49\textwidth}
	\includegraphics[width=1.1\textwidth]{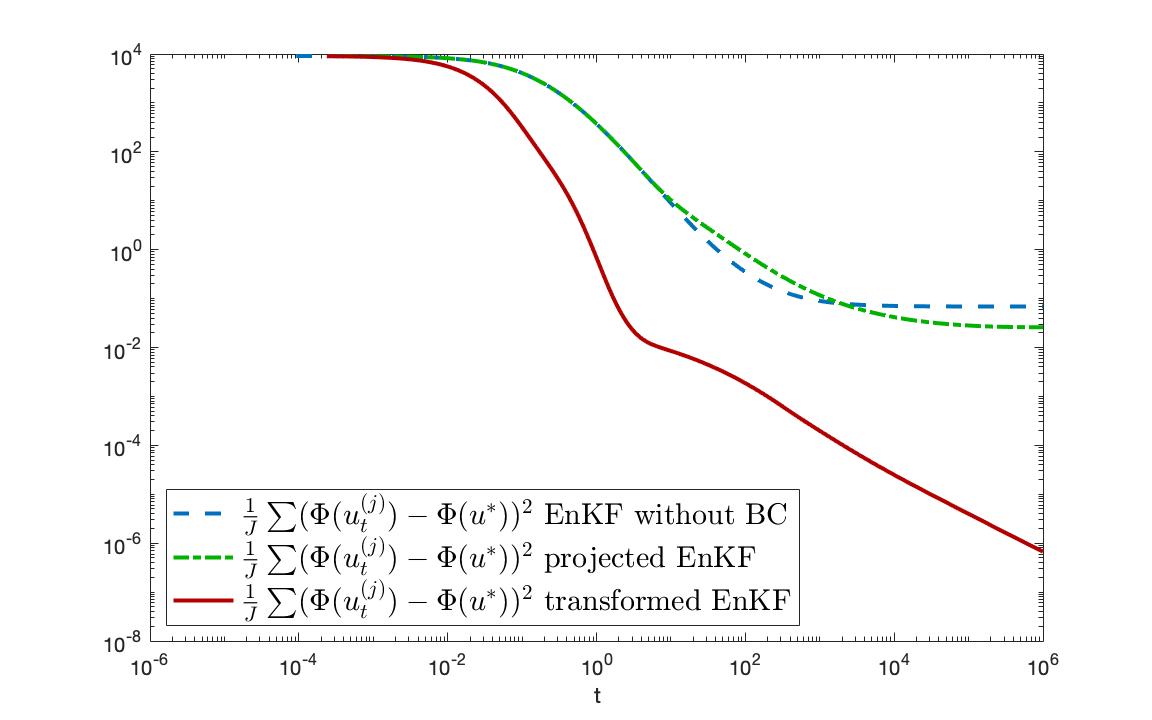}
	\end{subfigure}
    \caption{KKT-Residuals and difference of the misfit functional {and the global minimum} in the transformed EnKF in comparison to the projected EnKF. $J=5$ particles have been simulated.}\label{fig:A1A3_Tkkt}
\end{figure} 

Our first set of experiments for the linear PDE are shown in Figures \ref{fig:A1A3_Test} -  \ref{fig:A1A3_Tkkt}, where we assume that we have full observations. The left hand side image of Figure \ref{fig:A1A3_Test} compares the performance of the different methods at reconstructing the truth. As known with EnKF, it can exhibit a smooth reconstruction which is exactly seen. For the projected EnKF we notice a similar performance, however it takes into consideration the constraints as part of its reconstruction is on the boundary. However when analyzing the transformed EnKF, we see an improvement over previous methods, which is also evident from the comparison of the observations from the right hand image.

Figure  \ref{fig:A1A3_Tspread} shows the ensemble collapse of each method, which we see is achieved. We observe the spread behaves almost at an identical rate for the EnKF and projected EnKF. To highlight further the benefit of using the transformed version, Figure \ref{fig:A1A3_Tkkt} demonstrates this by showing a sharper decrease in both the KKT residuals and the difference of the misfit functional and global minimizer. As the projected method eventually levels flattens at $10^{-1}$ for both, the transformed version continues the achieve smaller differences.

\begin{figure}[!htb]
	\begin{subfigure}[c]{0.49\textwidth}
	\includegraphics[width=1.1\textwidth]{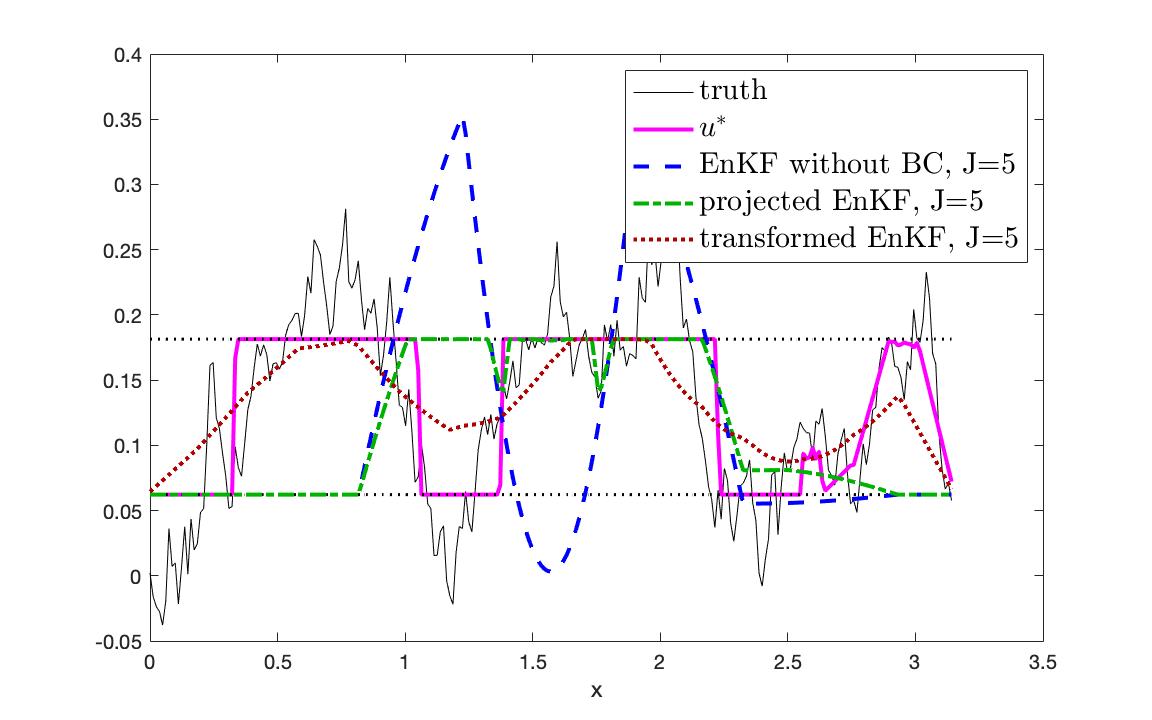}
	\end{subfigure}
	\begin{subfigure}[c]{0.49\textwidth}
	\includegraphics[width=1.1\textwidth]{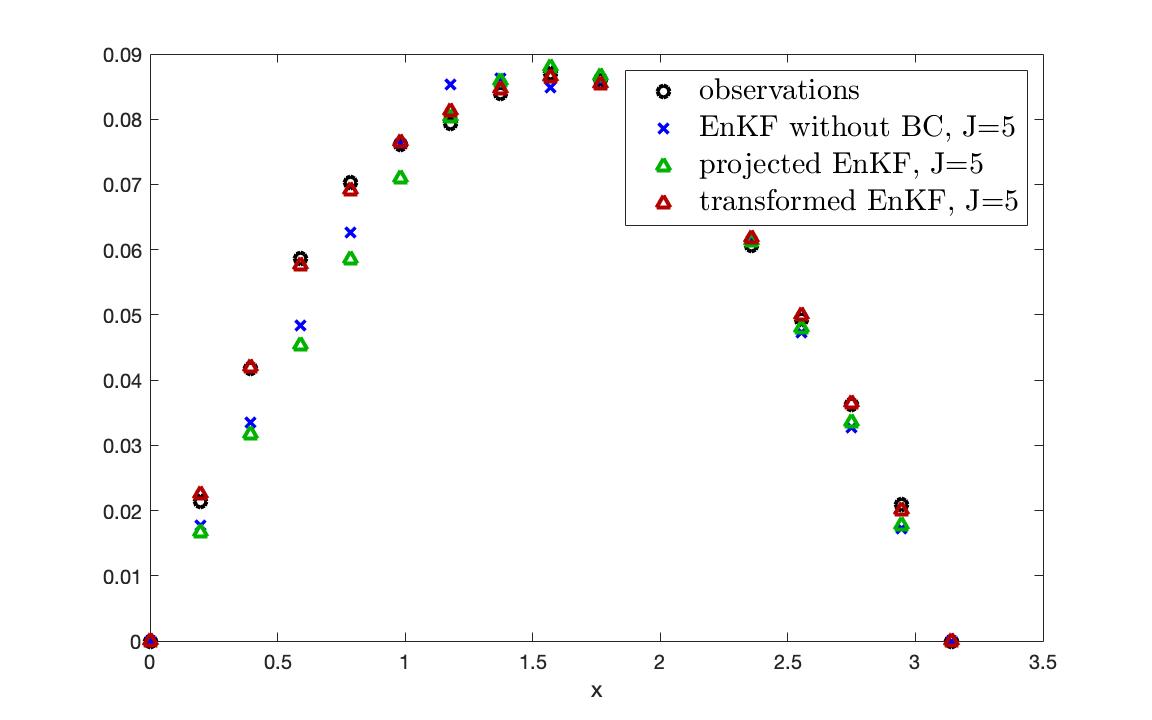}
	\end{subfigure}
    \caption{Transformed EnKF estimation in comparison to the EnKF estimation and the projected EnKF estimation. $J=5$ particles have been simulated.}\label{fig:A1A4_Test}
\end{figure} 

\begin{figure}[!htb]
	\includegraphics[width=0.49\textwidth]{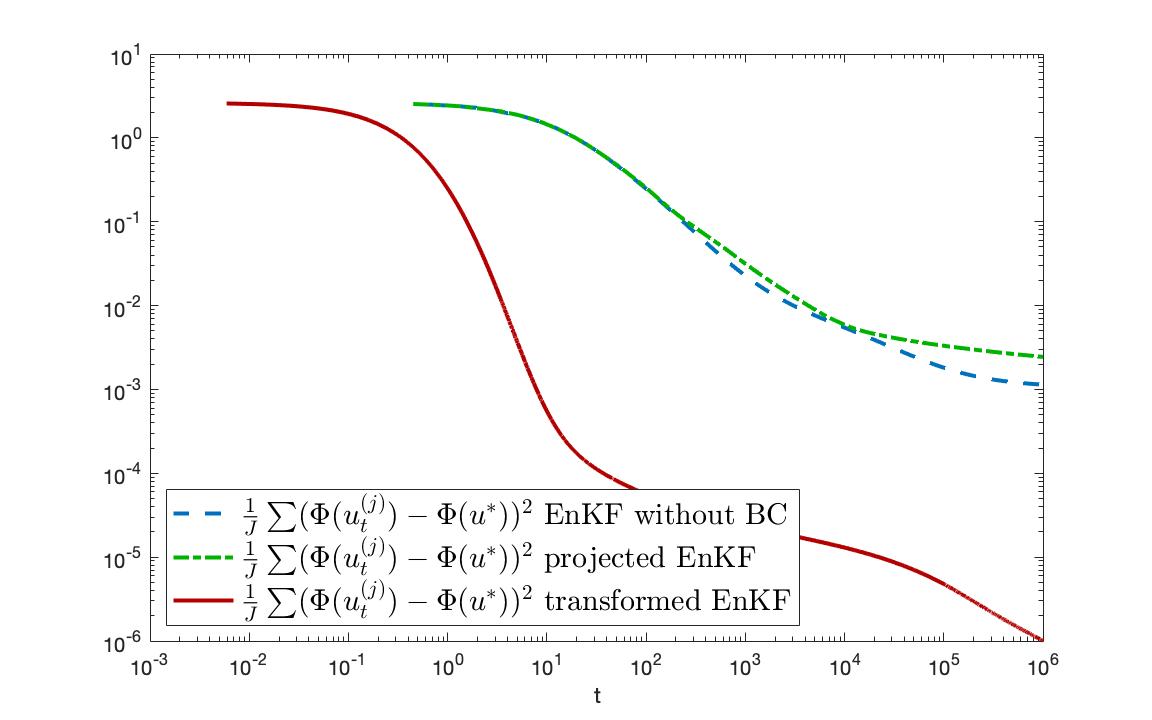}
    \caption{Difference of the misfit functional and the global minimum in the transformed EnKF in comparison to the EnKF and the projected EnKF. $J=5$ particles have been simulated.}\label{fig:A1A4_Tkkt}
\end{figure} 

The second set of experiments in the linear setting are shown in Figures \ref{fig:A1A4_Test} -  \ref{fig:A1A4_Tkkt}, where we assume that we have 15 low dimensional observations. The results obtained here are analogous to the full observation case, in that the transformed version outperforms the other methods.

\subsection{Non-linear inverse problem} {To demonstrate the effectiveness of the transformed projection, we consider a 2D nonlinear PDE. We will use an analogous nonlinear version of our linear elliptic PDE, which arises in subsurface flow. The forward problem associated with the Darcy flow model is using the permeability $\kappa \in L^{\infty}(D)$ to solve 
\begin{align}
\label{eq:darcy}
-\nabla\cdot({\kappa}\nabla p) &= f, \quad x \in  D, \\
\label{eq:bc2}
 p &=0, \quad x\in \partial D,
\end{align}
where $\nabla \cdot$ denotes the divergence and we have imposed zero boundary conditions \eqref{eq:bc2}. Our forward solver is based on a second-order centred finite difference method with mesh size $h=2^{-4}$. We take the source term of \eqref{eq:darcy} as $f=1$.
The inverse problem associated with \eqref{eq:darcy} , where we specify $\kappa=\exp(u)$, is to reconstruct $\kappa$ from noisy linear functionals 
\begin{equation}
\label{eq:func}
 y_k = l_k(\kappa) + \eta_k, \quad k=1,\cdots,K. 
\end{equation}
By defining $\mathcal{G}_k(u) = l_k(\kappa)$, we can rewrite \eqref{eq:func} as the inverse problem 
\begin{equation*}
y = \mathcal{G}(u) + \eta, \quad \eta \sim N(0,\Gamma).
\end{equation*}
The inverse Darcy flow equation has been studied, and is continuously used as an example which is well-represented for PDE-constrained Bayesian inversion. We stick to the same setting as in subsection \ref{ssec:linear}, but with the modifications of only having 16 low dimensional observations, and testing a system with $l=64$. Our prior $u \sim N(0,C)$ is simulated through a Karhunen-Lo\`{e}ve expansion of the form
\begin{equation}
\label{eq:kl}
u = \sum^{\mathcal{J}}_{j} \sqrt{\lambda_j}\xi_j \phi_j, \quad \xi \sim N(0,I),
\end{equation}
where $({\lambda_j},{\phi_j})$ is the eigenbasis of the covariance operator $C$, expressed as
$$
C := \sigma^2(I - \Delta u)^{-\nu},
$$
where $\sigma^2$ is a scaling constant, $\nu > d/2$ is the smoothness of the prior and $\Delta$ is the Laplace operator in 2D. The hyperparameters are specified as $\sigma^2=1$ and $\nu=2$. As suggested through Remark \ref{rem:vi} the projected EnKF with variance inflation can be viewed as a special case of the transformed version. 
}

{Since our analysis was in the linear setting we need to discuss how to incorporate variance inflation in the nonlinear setting. 

\subsubsection{Variance inflation in the nonlinear setting}
In the nonlinear case there is no obvious representation of the continuous time limit of the EnKF as preconditioned gradient flow. We consider the following approximation based on the Taylor expansion
\begin{equation*}
\cG(u^{(j)})-\cG(\bar u) \approx D\cG(\bar u) (u^{(j)}-\bar u),
\end{equation*}
where $D\cG(\bar u)$ is the Jacobian matrix of $\cG$ at $\bar u$. Hence, we will approximate the mixed sample covariance $C^{up}(u)$ by

\begin{equation*}
C^{up}(u) \approx \frac{1}{J} \sum\limits_{j=1}^J (u^{(j)}-\bar u)(u^{(j)}-\bar u)^\top D\cG(\bar u)^\top = C(u)D\cG(\bar u)^\top.
\end{equation*}
Note that we have introduced a second approximation $\bar\cG \approx \cG(\bar u)$. We will use variance inflation in sense of
\begin{equation}\label{eq:apprx_cup}
C^{up}(u) \approx C(u)D\cG(\bar u)^\top \mapsto (C(u)+\vartheta C_0)D\cG(\bar u)^\top,
\end{equation}
and we will write the EnKF with variance inflation through the following approximation
\begin{equation}
\label{eq:inf_limit2}
\frac{du^{(j)}}{dt} = (C^{up}(u)+\vartheta C_0 D\cG(\bar u)^T)(y-\cG(u^{(j)})).
\end{equation}

By application of this method of variance inflation one should mention that there exists the disadvantage of computing the derivative of the nonlinear map $\cG$. 

\begin{figure}[!htb]	
	\includegraphics[width=0.75\textwidth]{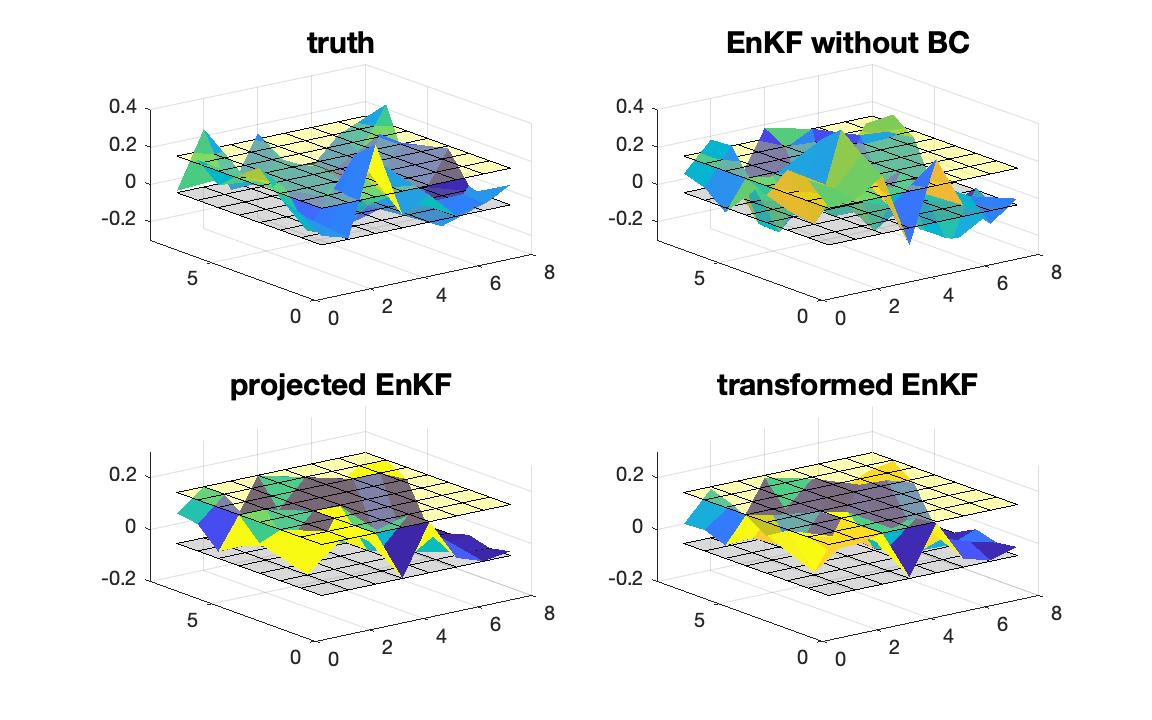}
	    \caption{Transformed EnKF parameter estimation in comparison to the EnKF estimation and the projected EnKF estimation. $J=5$ particles have been simulated.}\label{fig:non_A1A4_Test2}
\end{figure} 

\begin{figure}[!htb]	
	\includegraphics[width=0.75\textwidth]{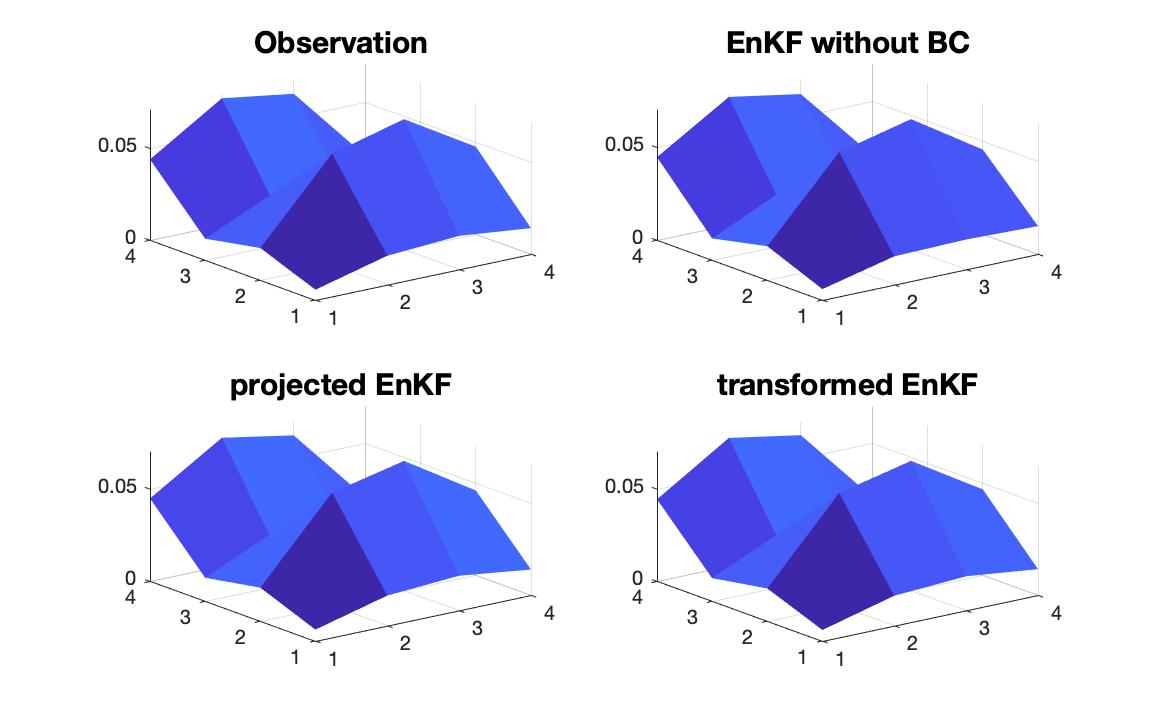}
    \caption{Transformed EnKF observation estimation in comparison to the EnKF estimation and the projected EnKF estimation. $J=5$ particles have been simulated.}\label{fig:non_A1A4_Testo2}
\end{figure} 

\begin{figure}[!htb]
	\begin{subfigure}[c]{0.49\textwidth}
	\includegraphics[width=1\textwidth]{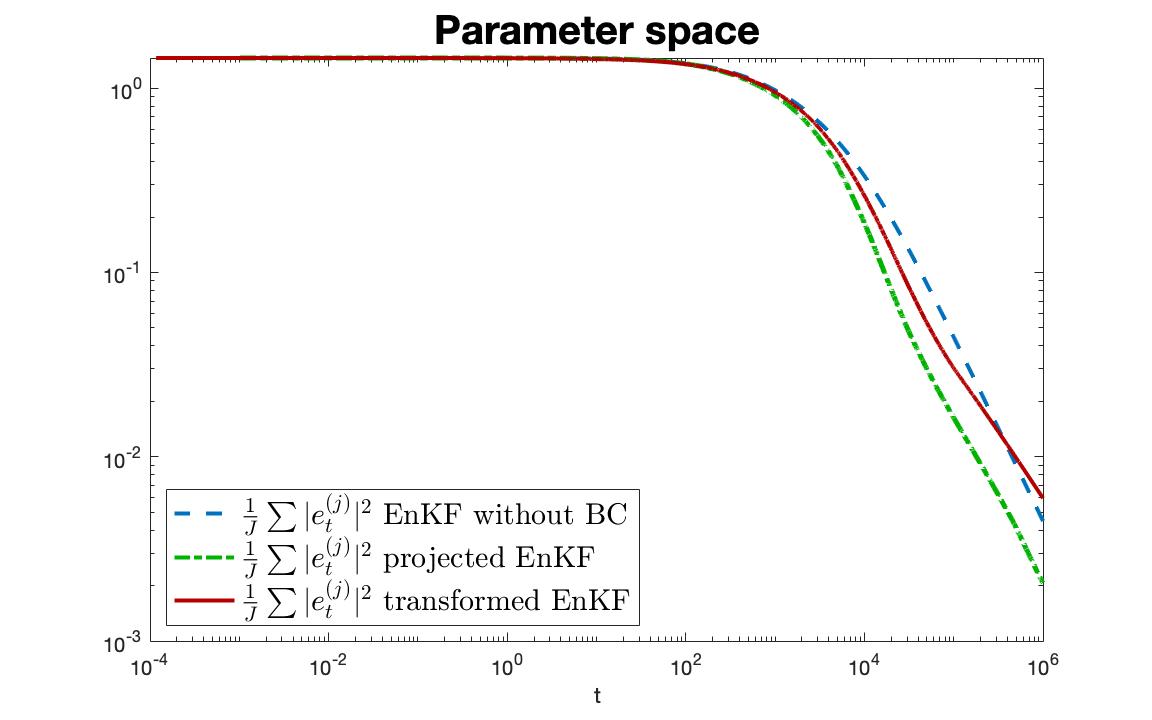}
	\end{subfigure}
	\begin{subfigure}[c]{0.49\textwidth}
	\includegraphics[width=1\textwidth]{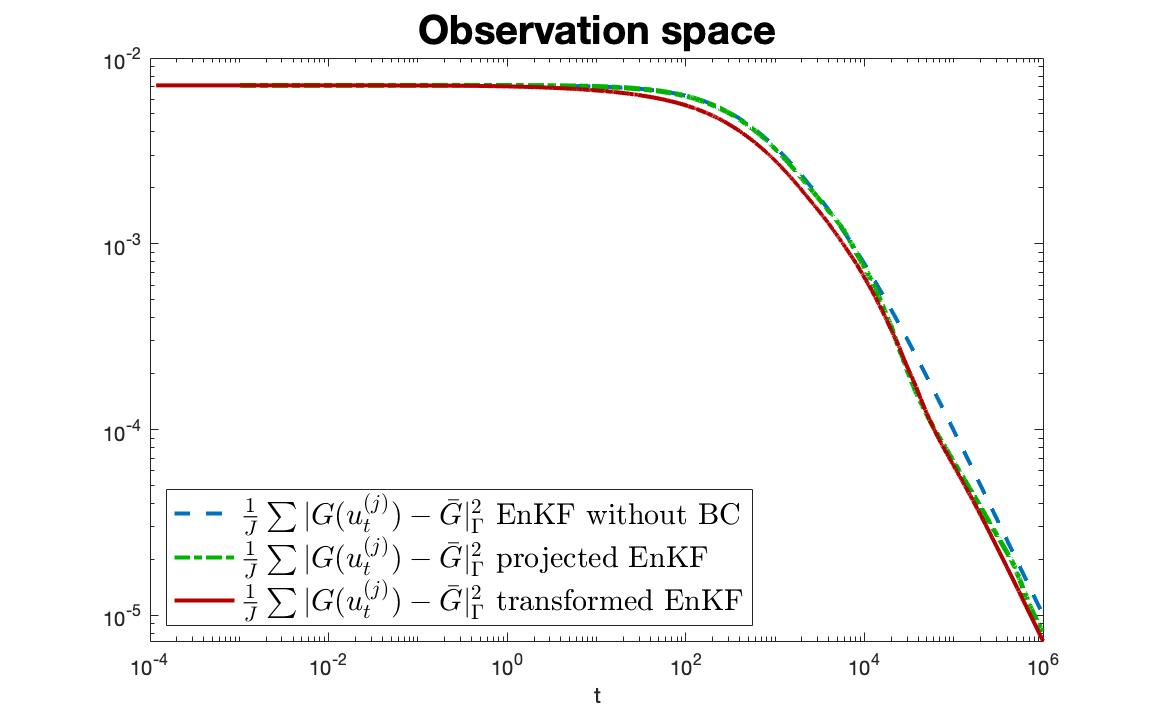}
	\end{subfigure}
    \caption{Ensemble spread in the transformed EnKF in comparison to the EnKF and the projected EnKF. $J=5$ particles have been simulated.}\label{fig:non_A1A4_Tspread2}
\end{figure} 

\begin{figure}[!htb]
	\includegraphics[width=0.49\textwidth]{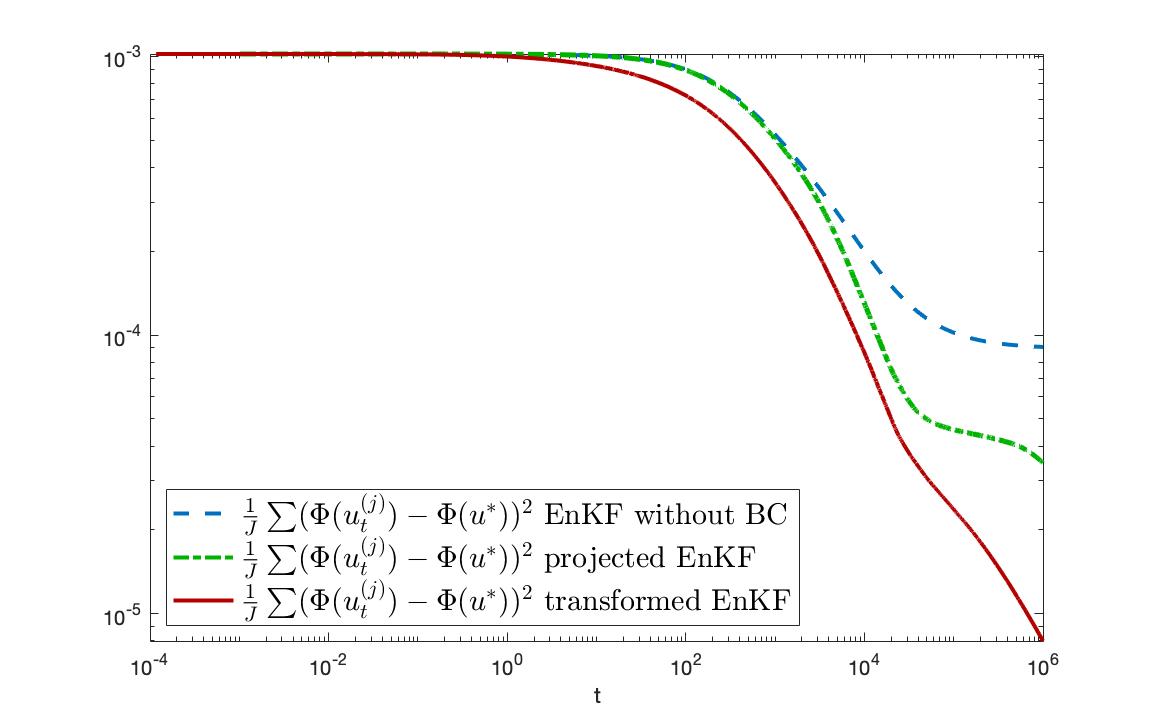}
    \caption{Difference of the misfit functional {and the global minimum} in the transformed EnKF in comparison to the projected EnKF. $J=5$ particles have been simulated.}\label{fig:non_A1A4_Tkkt2}
\end{figure} 
}

{For the non-linear experiments Figure \ref{fig:non_A1A4_Test2} shows the performance of each method w.r.t. to the truth. It can be seen that the EnKF without constraints does not remain within the feasible set unlike the other two methods, despite Figure  \ref{fig:non_A1A4_Testo2} looking identical across all methods. To see a more in-depth representation of the performance we analyze the ensemble spread seen in Figure \ref{fig:non_A1A4_Tspread2} where they all seem to converge to zero, where the rates look similar. However by looking at the difference of the misfit functional with the minimizer in Figure \ref{fig:non_A1A4_Tkkt2}, we see the difference of the transformed method continues to decrease while for the projected EnKF it starts to flatten, likewise with the original EnKF. This highlights further the benefit of using the transformed EnKF. }

\section{Conclusion}
\label{sec:conc}
{Our motivation behind this work was to gain insight in how constrained optimization could be implemented within ensemble Kalman inversion (EKI). To accommodate our algorithm of choice, we made use of the formulated projected Newton method discussed by Bertsekas et al. \cite{DPB82}. This was achievable as the projected Newton method could be related through the least-squares formulation. As a result a box-constrained optimization method was adapted for EKI. This included deriving a continuum limit and a gradient flow structure. The key insight from this work is that the projected EKI with box constraints does not always result in a correct descent direction to the minimizer. By tempering with the preconditioner through additive variance inflation it was shown that the correct descent direction could be acquired. This was shown analytically through a convergence analysis and various numerical experiments. The results presented in this work can also generalize to the ESRF.

As EKI can be interpreted as an optimizer, this encourages further research into adopting tools from optimization theory. The current article provides a basis into why one would be interested in using these approaches, however there remains numerous other avenues. These include trying to quantify the ensemble Kalman filter for inversion as an optimizer as well as implementing some form of gradient descent techniques \cite{CLTZ16}.  Another direction to consider are hierarchical inverse problems. A hierarchical version of EKI was proposed in \cite{CIRS17}, where the hyperparameters were governed through a uniform distribution. As there is no guarantee the parameters will stay within the range of distribution, one could apply the techniques in this paper to a hierarchical setting.
Finally in the non-linear setting gradients are required to approximate the continuum limits with variance inflation. As this procedure can be costly in high dimensions, we hope to improve on this burden with various computational techniques. }

\section*{Acknowledgements} 
ClS would like to thank the Isaac Newton Institute for Mathematical Sciences for support and hospitality during the programme Uncertainty quantification for complex systems: theory and methodologies when work on this paper was undertaken. NKC acknowledges a Singapore Ministry of Education Academic Research Funds Tier 2 grant [MOE2016-T2-2-135]. SW is grateful to the DFG RTG1953 ``Statistical Modeling of Complex Systems and Processes” for funding of this research.

\bibliographystyle{siam}
\bibliography{mybib}

\end{document}